\documentclass[aap,seceqn,MSNbibl,citesort,dvips]{arximspdf}
\usepackage{graphicx}

\doi{10.1214/10-AAP706}
\volume{21}
\issue{2}
\pubyear{2011}
\firstpage{609}
\lastpage{644}

\makeatletter

\newcommand{\eqref}[1]{(\ref{#1})}

\newcommand{\lVert}{\Vert}
\newcommand{\rVert}{\Vert}

\newcommand{\lvert}{\vert}
\newcommand{\rvert}{\vert}

\newtheorem{theorem}{Theorem}
\newtheorem{prop}{Proposition}
\newtheorem{lm}{Lemma}
\newtheorem{cor}{Corollary}

\newproclaim{example}{Example}
\newproclaim{Assumption}{Assumption}

\newcommand{\R}{{\mathbb R}}
\newcommand{\RR}{{\mathcal R}}
\newcommand{\F}{{\mathcal F}}
\newcommand{\PP}{{\mathbb P}}
\newcommand{\QQ}{{\mathbb Q}}
\newcommand{\E}{{\mathbb E}}

\newcommand{\mfA}{{\mathfrak A}}
\newcommand{\mfa}{{\mathfrak a}}
\newcommand{\mfC}{{\mathfrak C}}
\newcommand{\mfc}{{\mathfrak c}}
\newcommand{\mfd}{{\mathfrak d}}
\newcommand{\mfe}{{\mathfrak e}}
\newcommand{\mfD}{{\mathfrak D}}
\newcommand{\mfF}{{\mathfrak F}}
\newcommand{\mfg}{{\mathfrak g}}
\newcommand{\mfG}{{\mathfrak G}}
\newcommand{\mfL}{{\mathfrak L}}
\newcommand{\mfN}{{\mathfrak N}}
\newcommand{\mfn}{{\mathfrak n}}
\newcommand{\mfP}{{\mathfrak P}}
\newcommand{\mfp}{{\mathfrak p}}
\newcommand{\mfQ}{{\mathfrak Q}}
\newcommand{\mfq}{{\mathfrak q}}
\newcommand{\mfR}{{\mathfrak R}}
\newcommand{\mfr}{{\mathfrak r}}
\newcommand{\mfS}{{\mathfrak S}}
\newcommand{\mfs}{{\mathfrak s}}
\newcommand{\mfU}{{\mathfrak U}}

\makeatother

\begin{document}
\begin{frontmatter}

\title{Hybrid Atlas models}
\runtitle{Hybrid Atlas models}

\begin{aug}
\author[A]{\fnms{Tomoyuki} \snm{Ichiba}\corref{}\ead
[label=e1]{ichiba@pstat.ucsb.edu}},
\author[B]{\fnms{Vassilios} \snm{Papathanakos}\ead
[label=e2]{ppthan@enhanced.com}},
\author[B]{\fnms{Adrian} \snm{Banner}\ead[label=e3]{adrian@enhanced.com}},
\author[B]{\fnms{Ioannis} \snm{Karatzas}\ead
[label=e4]{ik@enhanced.com}\thanksref{T1}} and
\author[B]{\fnms{Robert} \snm{Fernholz}\ead[label=e5]{bob@enhanced.com}}
\runauthor{T. Ichiba et al.}
\affiliation{University of California, INTECH, INTECH, INTECH and INTECH}
\dedicated{Dedicated to Professor J. Michael Harrison on the occasion
of his 65th Birthday}
\address[A]{T. Ichiba\\
Department of Statistics\\
and Applied Probability \\
University of California \\
South Hall\\
Santa Barbara 93106 \\
USA\\
\printead{e1}}
\address[B]{V. Papathanakos\\
A. Banner\\
I. Karatzas\\
R. Fernholz\\
INTECH Investment Managment\\
One Palmer Square, Suite 441\\
Princeton, New Jersey 08642\\
USA\\
\printead{e2}\\
\phantom{E-mail: }\printead*{e3}\\
\phantom{E-mail: }\printead*{e4}\\
\phantom{E-mail: }\printead*{e5}}
\end{aug}

\thankstext{T1}{Supported by   NSF Grants DMS-06-01774 and DMS-09-05754.
This author is on partial leave from the Department of Mathematics,
Columbia University.}

% HISTORY:
\received{\smonth{9} \syear{2009}}
\revised{\smonth{3} \syear{2010}}

% ABSTRACT
%
\begin{abstract}
We study Atlas-type models of equity markets with local characteristics
that depend on both name and rank, and in ways that induce a stable
capital distribution. Ergodic properties and rankings of processes are
examined with reference to the theory of reflected Brownian motions in
polyhedral domains. In the context of such models we discuss
properties of various investment strategies, including the so-called
\textit{growth-optimal} and \textit{universal} portfolios.
\end{abstract}

% KEYWORDS
%

\end{frontmatter}

%s1 ###
\section{Introduction} \label{sec: Intro}

In modeling equity market behavior, the goal is to construct models
that are simple enough to be amenable to mathematical analysis, yet
complicated enough to capture the salient characteristics of real
equity markets. A particularly salient characteristic of an equity
market is its \textit{capital distribution curve},
%
%e1.1 ###
%
\begin{equation}
\label{LogLog}
\log k \mapsto\log\mu_{(k)} (t) ,\qquad k=1,\ldots, n ,
\end{equation}
that is, the logarithms of the individual companies' relative
capitalizations (market weights)
$ \mu_{(\cdot)}(t) $ at time $ t $, arranged in descending
order $ \mu_{(1)}(t) \ge\mu_{(2)}(t)\ge\cdots\ge\mu_{(n)}(t) $,
versus the logarithms of their respective ranks from the largest
company $ k=1 $ down to the smallest $ k = n $.

The capital distribution curve for the US equity market has shown
remarkable stability over the last century (see, for instance, Figure
5.1 of Fernholz \cite{F02}), and this stability has been captured in
the Atlas and first-order models introduced in \cite{F02} and studied
by Banner, Fernholz and Karatzas \cite{BFK05} and others. These models
assign growth rates and volatilities to the different stocks based
purely on the stocks' rank in terms of relative capitalization at any
given time, and roughly speaking,
if the smallest stocks are assigned big growth rates and big variances,
then a stable capital distribution does indeed emerge.

While Atlas and first-order models are able to reproduce the shape and
stability of the capital distribution curve, they still fail to provide
an accurate representation of market behavior. It was shown in
\cite{BFK05} that in these models each stock spends
about the same proportion of time in each rank over the long term.
While this kind of ergodicity may be a nice mathematical property, it
does not seem to hold for real markets: in real markets the largest
stocks seem to retain their status for long periods of time, while most
stocks never reach the upper echelons of capitalization. Hence, a more
elaborate model is needed.

In this paper we generalize the first-order models by introducing
name-based effects of companies, in addition to the purely rank-based
effects of the simpler models studied in \cite{BFK05}. The resulting
\textit{hybrid model} (\ref{eq: model}) has more flexibility to describe
faithfully the complexity of the entire market; in particular, the
model has both stability properties and occupation time properties that
are realistic.

%%%%%%%%%%%%%%%%%%%%%%%%%%%%%%%%%%%%%%%%%%%%%%%%%%%%%%%%%%%%%%%%%%%%%%%%%%%%%%%%%%%%%%%%%
\subsection*{Relation to extant literature}
From a different point of view,
the Atlas model can be seen as a physical particle system with each
company represented by a particle diffusing on the positive real line.
These individual diffusive motions have drift and volatility
coefficients that depend on the entire configuration of particles at
any given moment, but not on the individual particles' ``identities''
(tags). Recently, Pal and Pitman \cite{PP08} and Chaterjee and Pal
\cite{CP07,CP09} studied such systems, specifically when the
drift coefficient is a function of the particle's rank and all
volatility coefficients are equal to a given constant. Under
appropriate conditions on the drift coefficients, the system has a
unique invariant probability measure in a lower-dimensional space; to
wit, the system of the $ n $ particles is itself not ergodic, but
the projected system in a lower-dimensional hyperplane turns out to be
ergodic, and with invariant probability
measure that has an explicit exponential-product-form probability
density function. Moreover, when the number of particles increases to
infinity, the system converges weakly to one described by a
Poisson--Dirichlet distribution on the real line. These analyses are
useful in studying the Atlas model for an equity market, when the
volatility coefficients are all equal.

The model is still tractable when its volatility coefficients depend on
the rankings. Questions of existence and uniqueness for such systems in
this generality are settled through the theory of martingale problems
studied by Stroock and Varadhan \cite{SV06} and Bass and Pardoux
\cite{BP87}. An important new feature of such models is that three or more
particles may now collide with each other at the same time with
positive probability, or even with probability one, under a suitably
``uneven'' volatility structure. This is a very significant departure
from the constant-volatility case. Some sufficient conditions on the
volatility coefficients for the occurrence and for the avoidance of
triple (or higher-order) collisions, are derived in \cite{IK09}, by
comparison with Bessel processes and with help from properties of
reflected Brownian motion.

The ranked particle system has a deep relation with the theory of
multi-dimensional reflected Brownian motion studied intensively in the
context of stochastic network systems by Harrison, Reiman and Williams
\cite{HR81,HW87Stoch,HW87AP} and their collaborators.
Recently, Dieker and Moriarty \cite{DM09} provided necessary and
sufficient conditions for the invariant density of semimartingale
reflected Brownian motions in a two-dimensional wedge to be expressed
as a finite sum of terms of product-of-exponential form, by extending
the geometric considerations on the so-called ``skew-symmetry''
condition. In the present paper we use this skew-symmetry condition
[see (\ref{eq: SSC}) in Lemma \ref{lm: vol cond 2}] for the
$n$-dimensional reflected Brownian motions to solve the basic adjoint
relation introduced in the context of a piece-wise constant drift
coefficient structure, and thus compute an invariant density for the
ranked process of the hybrid Atlas model as a
sum of products of exponentials. With this explicit formula we compute
the invariant distribution of the capital distribution curve as well as
the long-term average occupation times.

Another interesting system of ranked particles is Dyson's process of
noncolliding Brownian motions, which are the ordered eigenvalues of a
Brownian motion on the space of Hermitian matrices. Recent work by
Warren \cite{Warren07} constructs Dyson's process using Doob's
$h$-transform and Brownian motion in the Gelfand--Tsetlin cone, as an
extension of Dub\'edat's work \cite{Dubedat04} on the relation between
reflected Brownian motions on the wedge and a Bessel process of
dimension three. The (infinite) ranked particle systems also appear in
mean-field spin glass theory of mathematical physics. In another recent
development, Arguin and Aizenman \cite{AA08} analyze robust
quasi-stationary competing particle systems with overlapping
hierarchical structures where the Poisson--Dirichlet distribution
emerges as in \cite{PP08}. Instead of taking Dyson's process or the
spin glass theory as our model for rankings in equity markets, we
obtain the ranked particle system through a
general formula of Banner and Ghomrasni \cite{BG07} for continuous
semimartingales in the hybrid Atlas model.

%%%%%%%%%%%%%%%%%%%%%%%%%%%%%%%%%%%%%%%%%%%%%%%%%%%%%%%%%%%%%%%%%%%%%%%%%%%%%%%%%%%%%%%%%
\subsection*{Preview}
This paper follows the following structure. We describe our model in
Section \ref{sec: Model}, its lower-dimensional ergodic properties in
Section \ref{sec: Ergodicity}, the dynamics of its rankings in Section
\ref{sec: rankings}, its invariant measure and occupation times in
Section \ref{sec: Inv meas} and some portfolio analysis in its context
in Section \ref{sec: Portfolio Analysis}. In the \hyperref[sec: Appendix]{Appendix}
we prove some auxiliary results stated in the main sections.

%%%%%%%%%%%%%%%%%%%%%%%%%%%%%%%%%%%%%%%%%%%%%%%%%%%%%%%%%%%%%%%%%%%%%%%%%%%%%%%%%%%%%%%%%
\subsection*{Notation}
The following notions and notation are useful to describe rankings as
in \cite{BFK05}. We consider a collection $ \{ Q_{k}^{(i)}\}_{1\le i,
k\le n} $ of polyhedral domains in $ \R^{n} $, where $ y =
(y_{1}, \ldots, y_{n}) \in Q_{k}^{(i)} $ means that the coordinate
$ y_{i} $ is ranked $k$th among $ y_{1}, \ldots, y_{n} $, with
ties resolved in favor of the lowest index (or ``name''). Note that for
every index $i = 1, \ldots, n $ and rank $ k = 1, \ldots, n $,
we have the partition properties
$\bigcup_{\ell=1}^{n} Q_{\ell}^{(i)} = \R^{n} = \bigcup_{j=1}^{n}
Q_{k}^{(j)} $.

We shall denote by $ \Sigma_{n} $ the symmetric group of
permutations of $ \{1, \ldots, n\} $. For each permutation $
{\mathbf p} \in\Sigma_{n} $ we consider $ \RR_{{\mathbf p}} : = \bigcap
_{k=1}^{n} Q_{k}^{({\mathbf p}(k))} $, the polyhedral chamber consisting
of all points $ y \in\R^{n} $ such that $ y_{\mathbf{ p}(k)} $
is ranked $k$th among $ y_{1}, \ldots, y_{n} $, for every $ k =
1, \ldots, n $. The collection of polyhedral chambers $ \{ \RR
_{\mathbf p}\}_{{\mathbf p} \in\Sigma_{n}} $ is a partition of all of $
\R
^{n} $.

Since for each $ y \in\R^{n} $ there exists a unique $ {\mathbf p}
\in\Sigma_{n} $ such that $ y \in\RR_{{\mathbf p}} $ (because of
the way ties are resolved), we shall find it useful to define an
\textit{indicator map} $ \R^{n} \ni(x_1,\ldots, x_n )' =x \mapsto\mfp^{x}
\in\Sigma_{n} $ such that $ x_{\mfp^{x}(1)} \ge\cdots\ge x_{\mfp
^{x}(n)} $. In other words, $ \mfp^{x} (k) $ is the index of the
coordinate in the vector $x$ that occupies the $k$th rank among $
x_1,\ldots, x_n $.

When matrices and vectors are used, the vector norm $ \lVert x
\rVert:= ( \sum_{i=1}^{n} x_{i}^{2} )^{1/2} $ and the inner
product $ \langle x , y \rangle:= \sum_{i=1} x_{i} y_{i} =
x^{\prime} y $ for $ x , y \in\R^{n} $, where $ \prime$
stands for transposition, are defined in the usual manner. The gradient
$ \nabla$ and the Laplacian $ \Delta$ operators on the space
$ C^{2} $ of twice continuously differentiable functions are used
in Section \ref{sec: Inv meas}, as well as the notation $
C_{c}^{2}(\cdot) $ [resp.,
$ C_{b}^{2}(\cdot) $] for the spaces of twice continuously
differentiable functions which have compact support (resp., are bounded
functions).

%s2 ###
\section{Model} \label{sec: Model}
We shall study an equity market that consists of $ n $ assets
(stocks) with capitalizations $ \mathfrak{ X} (t) = ( X_{1}(t) ,
\ldots, X_{n}(t))' $ which are positive for all times $ 0 \le t <
\infty$. The random variable $ X_{i}(t) $ represents the
capitalization at time $ t $ of the asset with index (name) $ i $.

We shall assume that the log-capitalizations $ Y_{i}( t) := \log
X_{i}( t) $, $ i = 1, \ldots, n $, satisfy the system of
stochastic differential equations
%
%e2.1 ###
%
\begin{eqnarray} \label{eq: model}\qquad
d Y_{i}(t) &=& \Biggl( \sum_{k=1}^{n} g_{k} {\mathbf1}_{Q_{k}^{(i)}}
(Y(t)) + \gamma_{i} + \gamma\Biggr)\, d t + \sum_{j=1}^{n} \rho
_{i, j} \,d W_{j}(t) \nonumber\\[-8pt]\\[-8pt]
&&{} + \sum_{k=1}^{n} \sigma_{k} {\mathbf1}_{Q_{k}^{(i)}}
(Y(t) ) \, d W_{i}(t) ,\qquad Y_{i}(0) = y_{i},\qquad0 \le
t <\infty\nonumber
\end{eqnarray}
with given initial condition $ y =(y_{1} , \ldots, y_{n})^{\prime
} $. The process $ W(\cdot):= (W_{1}(\cdot) ,\ldots,\break
W_{n}(\cdot))^{\prime} $ is an $n$-dimensional Brownian motion. As
long as the $ n$-dimensional process $ Y(\cdot):= ( Y_{1}(\cdot
) , \ldots, Y_{n}(\cdot) )^{\prime} $ of log-capitalizations is
in the polyhedron $ Q_{k}^{(i)} $, the $i$th-coordinate $
Y_{i}(\cdot) $ is ranked $k$th among $ Y_{1}(\cdot) , \ldots,
Y_{n}(\cdot) $ and behaves like a Brownian motion with drift $ g_{k}
+ \gamma_{i} + \gamma$ and variance $ (\sigma_{k}+ \rho_{i,
i})^{2} + \sum_{j\neq i}\rho_{i, j}^{2} $.\vspace*{1pt} The constants $ \gamma
, \gamma_{i} $ and $ g_{k} $ represent respectively a common, a
name-based and a rank-based drift (growth rate) whereas the constants
$ \sigma_{k} $ and $ \rho_{i, j} $ represent rank-based
volatilities and name-based correlations, respectively.
\begin{Assumption*}
Throughout this paper we assume (without
loss of
generality) that the drift constants satisfy the stability condition
%
%e2.2 ###
%
\begin{equation} \label{eq: drift cond 1}
\sum_{k=1}^{n} g_{k} + \sum_{i=1}^{n} \gamma_{i} = 0 .
\end{equation}
We shall assume that the $ (n\times n) $ matrices
%
%e2.3 ###
%
\begin{equation} \label{eq: vol cond 1}\quad
\mfs_{\mathbf p} := \operatorname{diag} \bigl(\sigma_{{\mathbf
p}^{-1}(1)} ,
\ldots, \sigma_{{\mathbf p}^{-1}(n)} \bigr) + ( \rho_{i, j} )_{1\le
i, j \le n}
\mbox{ are positive definite}
\end{equation}
for every $ {\mathbf p} \in\Sigma_{n} $, with $ \sigma_{k} > 0 $
for every $ k = 1, \ldots, n $.
\end{Assumption*}

Equation (\ref{eq: model}) can be cast in vector form as
%
%e2.4 ###
%
\begin{equation}
\label{eq: model matrix}
d Y(t) = G (Y(t)) \,d t + S (Y(t)) \,d W(t) ,\qquad Y(0) = y
\in\R^{n}
\end{equation}
for $ 0 \le t < \infty$, where the functions $ G \dvtx\R
^{n}\rightarrow\R^{n} $ and $ S \dvtx\R^{n} \mapsto\R^{n\times n}
$ are
\begin{eqnarray*}
G(y) &:=& \sum_{{\mathbf p} \in\Sigma_{n}} {\mathbf1}_{\RR_{{\mathbf
p}}}(y) \cdot
\bigl( g_{{\mathbf p}^{-1}(1)} + \gamma_{1} + \gamma, \ldots,
g_{{\mathbf p}^{-1}(n)} + \gamma_{n} + \gamma\bigr)^{\prime} , \\
S(y) &:=& \sum_{{\mathbf p} \in\Sigma_{n}} {\mathbf1}_{\RR_{{\mathbf
p}}}(y) \cdot
\mfs_{{\mathbf p}},\qquad y \in\R^{n}.
\end{eqnarray*}
Thus (\ref{eq: model}) is a system of stochastic differential equations
with coefficients that are piecewise constant, the same in each
polyhedral chamber $ \RR_{{\mathbf p}} $, $ {\mathbf p} \in\Sigma_{n}
$. Under the assumption of positive definiteness in (\ref{eq: vol cond
1}), the system (\ref{eq: model}) admits a weak solution $ (Y, W) $
on a filtered probability space $ (\Omega, \F, \{\F_{t}\} , \PP
) $ satisfying the usual conditions. By the martingale-problem theory
of Stroock and Varadhan \cite{SV06} and the results in Bass and Pardoux
\cite{BP87}, this weak solution is unique in the sense of the
probability distribution.

%%%%%%%%%%%%%%%%%%%%%
%s3 ###
\section{Ergodicity} \label{sec: Ergodicity}
%%%%%%%%%%%%%%%%%%%%%
Thanks to assumption (\ref{eq: drift cond 1}) on the drifts, and taking
the average of both sides of (\ref{eq: model}), we obtain the average
log-capitalization process $ \overline{Y} (\cdot) := \sum_{i=1}^{n}
Y_{i}(\cdot) / n $ in the form
%
%e3.1 ###
%
\begin{eqnarray}
\label{eq: 2.4a}
&&\overline{Y} (t) = \frac{1}{n}\sum_{i=1}^{n} {y}_{i} + \gamma t
+ \frac{1}{n}\sum_{k=1}^{n} \sigma_{k} B_{k}(t) + \frac{1}{n} \sum_{i,
j=1}^{n} \rho_{i, j} W_{j}(t) , \nonumber\\[-8pt]\\[-8pt]
&&\qquad\mbox{where } B_{k}(t) := \sum_{i=1}^{n} \int^{t}_{0}
\mathbf{1}_{Q_{k}^{(i)}}(Y(s)) \,d W_{i}(s) \mbox{, }k = 1, \ldots,
n, \nonumber
\end{eqnarray}
for $ 0 \le t < \infty$ because of $\bigcup_{i=1}^{n} Q_{k}^{(i)}
= \R^{n} $. Here $ B_{1}(\cdot) , \ldots, B_{n}(\cdot) $ are
continuous local martingales with quadratic (cross-)variations $
\langle B_{k} , B_{\ell} \rangle$ $ (t) = t \delta_{k, \ell} $,
and hence are independent standard Brownian motions by the Knight
theorem. It follows that the average $ \overline{Y }(\cdot) $ of
the log-capitalizations $ Y_{1}(\cdot) , \ldots, Y_{n}(\cdot) $
grows at a rate equal to the common drift $ \gamma$, that is,
%
%e3.2 ###
%
\begin{equation}
\label{eq: average Y LLN}
\lim_{T\to\infty} \frac{ \overline{Y } (T) }{T} = \gamma
\qquad\mbox{holds a.s.,}
\end{equation}
by the strong law of large numbers for Brownian motion.

In order to study the long-term behavior of the whole
log-capitalizations, let us quote Theorems 4.1 and 5.1 on pages 119--121
of Khas'minskii \cite{K80}, since our argument relies on them rather decisively.
\begin{prop}[(Khas'minskii)] \label{prop: khas} Consider a diffusion $
\xi(\cdot) $ with values in a subset $ E $ of
Euclidean space. Assume that there exists a bounded domain $ U
\subset E $ with regular boundary, having the following properties:
\begin{enumerate}[(B.2)]
\item[(B.1)] In the domain $ U $ the smallest eigenvalue of the
diffusion matrix of the process $ \xi(\cdot) $ is bounded away from zero.
\item[(B.2)] If $ x \in E \setminus
U $, the mean time $ \tau$ at which a path issuing from $ x $
reaches the set $ U $ is finite, and $ \sup_{x \in K} \E_x (\tau)
< \infty$ for every compact subset $ K \subset E $.
\end{enumerate}
Then the Markov process $ \xi(\cdot) $ has a unique stationary
distribution $ \mu$, and which satisfies the Strong Law of
Large Numbers
\[
\PP_x \biggl( \lim_{T\to\infty} \frac{1}{T} \int^T_0 f(\xi(t)) \,d t =
\int_E f(y) \mu(dy)\biggr) = 1,\qquad x \in E ,
\]
for any bounded, measurable function $ f\dvtx E \to\R$.
\end{prop}

Let us introduce the column vector $ {\mathfrak1} := (1, \ldots,
1)^{\prime} $ and the subspace
\[
\Pi:= \{ y \in\R^{n} \vert{\mathfrak1}^{\prime} y = 0 \}.
\]
\begin{theorem} \label{prop: 1} In addition to (\ref{eq: drift cond
1}) and (\ref{eq: vol cond 1}), let us impose for every $ {\mathbf p}
\in
\Sigma_{n} $ the following stability condition:
%
%e3.3 ###
%
\begin{equation}
\label{eq: drift cond 2}
\sum_{k=1}^{\ell} \bigl( g_{k} + \gamma_{{\mathbf p} (k)} \bigr) < 0
,\qquad\ell= 1, \ldots, n-1 .
\end{equation}
Then the deviations $ \widetilde Y(\cdot) := ( Y_{1}(\cdot) -
\overline{Y}(\cdot), \ldots, Y_{n}(\cdot) - \overline{Y}(\cdot)
) $ of the log-capitaliza\-tions $ Y_{1}(\cdot) , \ldots,
Y_{n}(\cdot) $ from their average are \textup{stable in distribution}:
there exists a unique invariant probability measure $ \mu$ for the
$ \Pi$-valued Markov process $ \widetilde Y(\cdot) $, and for any
bounded, measurable function $ f\dvtx\Pi\rightarrow\mathbb{R} $ we
have the Strong Law of Large Numbers
%
%e3.4 ###
%
\begin{equation}
\label{SLLN}
\lim_{T \rightarrow\infty} \frac{ 1 }{ T } \int_0^T f (
\widetilde{Y} (t) ) \,dt = \int_\Pi f (y) \mu(dy) ,\qquad
\mbox{a.s.}
\end{equation}
\end{theorem}
\begin{pf}
From (\ref{eq: model}) and (\ref{eq: 2.4a}), we have
%
%e3.5 ###
%
\begin{equation} \label{eq: y tilde}
d \widetilde Y (t) = \widetilde G( \widetilde Y(t)) \,d t +
\widetilde S (\widetilde Y(t)) \,d W(t), \qquad
\widetilde Y(0) = \widetilde y ,
\end{equation}
where $ \widetilde y:= y - {\mathfrak1}^{\prime} y \cdot{\mathfrak
1} / n $, $ \widetilde G (y) := G(y) - \gamma\cdot
{\mathfrak1} $ and $ \widetilde S(y) := S(y) - {\mathfrak1}
{\mathfrak1}^{\prime} S(y) / n $ for $ y \in\R^{n} $. By
(\ref{eq: vol cond 1}) the covariance matrix in (\ref{eq: y tilde}) is
uniformly nondegenerate: for all $ x, y \in\Pi$ we have
\[
x^{\prime} \widetilde S(y) x = x^{\prime} S(y) x - x^{\prime}
{\mathfrak1} {\mathfrak1}^{\prime} S(y) x / n = x^{\prime}
S(y) x = \sum_{{\mathbf p} \in\Sigma_{n} } {\mathbf1}_{\RR_{\mathbf p}}(y)
\cdot x^{\prime} \mfs_{\mathbf p} x
\]
and
%
%e3.6 ###
%
\begin{equation} \label{eq: bounds S tilde}
\lambda_{0} \lVert x \rVert^{2} \le x^{\prime} \widetilde S(y) x \le
\lambda_{1} \lVert x
\rVert^{2} ,
\end{equation}
where $ \lambda_{0} (\lambda_{1}) $ are the minimum (maximum) of
the smallest (largest) eigenvalues of the positive definite matrices
$ \mfs_{\mathbf p} $ over $ {\mathbf p} \in\Sigma_{n} $ in (\ref{eq:
vol cond 1}).

Summation-by-parts, along with (\ref{eq: drift cond 1}) and (\ref{eq:
drift cond 2}), lead now to
%
%e3.7 ###
%
\begin{eqnarray} \label{eq: bound y G}\qquad
y^{\prime} \widetilde G(y) &=& \sum_{i=1}^{n} y_{i} \bigl( g_{({\mfp
^y})^{-1}(i)} + \gamma_{i}\bigr) = \sum_{k=1}^{n} y_{{\mfp^y}(k)}
\bigl( g_{k} + \gamma_{{\mfp^y}(k)} \bigr)\nonumber\\
&=& y_{{\mfp^y}(n)} \underbrace{\sum_{k=1}^{n} \bigl( g_{k} + \gamma
_{{\mfp^y}(k)} \bigr)}_{= 0} + \sum_{k=1}^{n-1} \bigl(y_{{\mfp^y}(k)}
- y_{{\mfp^y}(k+1)} \bigr) \sum_{\ell=1}^{k} \bigl(g_{\ell} + \gamma
_{{\mfp^y}(\ell)} \bigr) \\
&\le& c \sqrt{n} \sum_{k=1}^{n} \bigl( y_{{\mfp^y}(k)} - y_{{\mfp
^y}(k+1)} \bigr) \le c \lVert y \rVert< 0,\qquad
y \in\Pi\cap\RR_{\mathbf p} ,\nonumber
\end{eqnarray}
where $ c := n^{-1/2} \max_{1\le\ell\le n-1 , {\mathbf p} \in\Sigma
_{n}} \sum_{k=1}^{\ell} (g_{k}+ \gamma_{{\mathbf p}(k)}) < 0 $. In the
last inequality
we have used for $ \mathbf{ p} \in\Sigma_n $ and $ y \in\Pi
\cap\RR_{\mathbf p} $ the properties $ y_{\mathbf{ p}(1)} \ge
y_{\mathbf{ p}(2)} \ge\cdots\ge y_{\mathbf{ p}(n)} $, thus also
$ y_{\mathbf{ p}(1)} \ge0 \ge y_{\mathbf{ p}(n)} $ and
\[
\lVert y \rVert^{2} \le n \max\bigl( y_{{\mathbf p}(1)}^{2} , y_{{\mathbf
p}(n)}^{2} \bigr) \le n \bigl(y_{{\mathbf p}(1)} - y_{{\mathbf p}(n)}
\bigr)^{2} .
\]

Now we consider the one-dimensional process $ N(t):= f(\widetilde
Y(t)) $ with $ f(y) = ( \lVert y \rVert^{2} + 1)^{1/2} > \lVert y
\rVert$ for $ y \in\Pi$. An application of It\^o's rule gives
\begin{eqnarray*}
d N(t) &=& \widetilde f(\widetilde Y(t)) \,d t + [ f(y)
]^{-1} y^{\prime} \widetilde S(y) \vert_{y = \widetilde Y(t)} \,d
W(t) ,\qquad 0 \le t < \infty, \\
\widetilde f(y) :\!&=& (f(y))^{-1} \bigl( y^{\prime} \widetilde
G(y) + \tfrac{1}{2} \operatorname{trace} ( \widetilde S(y) \widetilde
S(y)^{\prime} ) \bigr) - \tfrac{1}{2}( f(y))^{-3} y^{\prime}
\widetilde S \widetilde S(y)^{\prime} y
\end{eqnarray*}
for $ y \in\Pi$. It follows from (\ref{eq: bounds S tilde}), (\ref
{eq: bound y G}) and the boundedness of $ \widetilde S (\cdot) $
that there exists a constant $ \kappa> 0 $ such that $
\widetilde f(y) \le c / 2 < 0 $ for $ \lVert y \rVert>
\kappa$. The diffusion coefficient $ [f(y)]^{-1} y^{\prime}
\widetilde S(y) $ of $ N(\cdot) $ is a vector whose entries are
uniformly bounded by some constants from
(\ref{eq: bounds S tilde}).

Thus $ N(\cdot) $ is positive recurrent with respect to the
interval $ (0, \kappa) $, and hence so is $ \widetilde Y(\cdot)
$ with respect to $ B \cap\Pi$ for some ball $ B \subset\R
^{n} $ centered at the origin.

Finally, we check the conditions (B.1) and (B.2) of Proposition \ref
{prop: khas}. For our diffusion $ \xi(\cdot) = \widetilde{Y}(\cdot)
$ on $ E = \Pi$ we have verified (B.1) in (\ref{eq: bounds S
tilde}). Assumption (B.2) is verified from the positive recurrence of
$ \widetilde{Y}(\cdot) $ with respect to $ U = B \cap\Pi$.
Therefore, by Proposition \ref{prop: khas}, we obtain the existence of
a unique invariant probability measure $ \mu$ that satisfies (\ref{SLLN}).
\end{pf}

Condition (\ref{eq: drift cond 2}) ensures that, if $ y_1 < y_2 <
\cdots< y_n $ and one subdivides at time $ t=0 $ the ``cloud'' of
$n$ particles diffusing on the real line according to the dynamics of
(\ref{eq: model}), into two ``subclouds''---one consisting of the $
\ell$ leftmost, and the other of the $ n-\ell$ rightmost,
particles---the two subclouds will eventually merge. They will not
continue to evolve like separate galaxies, that never make contact with
each other (cf. the Remark following Theorem 4 in Pal and Pitman
\cite{PP08} for an elaboration of this point in the case of the purely
rank-based first-order model with equal variances).
\begin{cor} \label{cor: aot 1} Under the assumptions of Theorem \ref
{prop: 1}, the \textup{long-term average occupation time} that company $
i $ spends in the $k$th rank, that is,
%
%e3.8 ###
%
\begin{equation} \label{eq: occup time 1}
\theta_{k, i} := \lim_{T\to\infty} \frac{1}{T} \int^{T}_{0} \mathbf{
1}_{Q_{k}^{(i)}}(\mathfrak{ X}(t)) \,d t,\qquad i, k = 1, \ldots
, n ,
\end{equation}
exists almost surely in $ [0, 1] $.

The resulting array of numbers $ \theta_{k, i} \in[0, 1] $ satisfy
$ \sum_{j=1}^{n} \theta_{k, j} = \sum_{\ell=1}^{n} \theta_{\ell,i} =
1 $ for each ``name'' $ i = 1, \ldots, n $ and ``rank'' $ k = 1,
\ldots, n $; that is, $ \vartheta:= (\theta_{k, i})_{1\le k, i \le
n} $ is a doubly stochastic matrix. Similarly, the average occupation
time $ \theta_{\mathbf{ p}} $ of the market in the polyhedral
chamber $ \RR_{\mathbf p} $, namely,
%
%e3.9 ###
%
\begin{equation} \label{eq: theta p def}
\theta_{\mathbf p} := \lim_{T\to\infty}\frac{1}{T} \int^{T}_{0}
{\mathbf1}_{\RR_{\mathbf p}} (\mathfrak{ X}(t)) \,d t\qquad
\mbox{exists a.s. in $ [0, 1] $}
\end{equation}
for every $ {\mathbf p} \in\Sigma_{n} $, and we have $ \theta_{k, i}
= \sum\theta_{\mathbf p} $, where the summation is over the set
$ \{ {\mathbf p} \in\Sigma_{n} \vert{\mathbf p}(k) = i\} $ of
permutations for $ 1 \le i, k \le n $.
\end{cor}

Indeed, by Theorem \ref{prop: 1} and in particular (\ref{SLLN}), the
quantity of (\ref{eq: occup time 1}) satisfies
\[
\theta_{k, i} = \lim_{T\to\infty} \frac{1}{T} \int^{T}_{0}
\mathbf{1}_{Q_{k}^{(i)}}(\mathfrak{ X}(t)) \,d t =
\lim_{T\to\infty} \frac{1}{T} \int^{T}_{0} {\mathbf1}_{Q_{k}^{(i)} \cap
\Pi}(\widetilde Y(t)) \,d t = \mu\bigl(Q_{k}^{(i)}\bigr) ,
\]
where $ \mu$ is the unique invariant probability measure for the
process $ \widetilde Y(\cdot) $ of (\ref{eq: y tilde}).
Since $ \bigcup_{\ell=1}^{n} Q_{\ell}^{(i)} = \R^{n} = \bigcup
_{j=1}^{n} Q_{k}^{(j)} $, it is obvious that $ \sum_{\ell
=1}^{n}\theta_{\ell, i} = \sum_{j=1}^{n} \theta_{k, j} = 1 $ for $
1 \le i, k \le n $. Equation (\ref{eq: theta p def}), and the claim
following it, are obtained similarly.

%%%%%%%%%%%%%%%%%%%%%%%%%%%%%%%%%%%%%%%%%%%%%%%%%%%%%%%%%%%%%%%%%%%%%%%%%%%%%%%%%%%%%%%%%
%s4 ###
\section{Rankings} \label{sec: rankings}
Let us now look at the log-capitalizations of the various companies
listed according to rank, namely
%
%e4.1 ###
%
\begin{equation}
\label{eq: rank proc 1}
Z_{k}(t) := \sum_{i=1}^{n} {\mathbf1}_{Q_{k}^{(i)}}(Y(t)) \cdot
Y_{i}(t) ,\qquad k = 1, \ldots, n , 0 \le t < \infty.
\end{equation}
These are the order statistics $ Z_{1} (\cdot) \ge\cdots\ge
Z_{n}(\cdot) $ for the log-capitalizations
$ Y_{1}(\cdot)= \log X_1(\cdot),\ldots,
Y_{n}(\cdot) = \log X_n(\cdot) $, listed from largest down to
smallest. We recall the indicator map $ \mfp^{x} $ introduced at
the end of Section \ref{sec: Intro}, and define the $ \Sigma
_{n}$-valued \textit{index process }
\[
\mfP_{t} := \mfp^{\mathfrak{ X}(t)} = \mfp^{Y(t)} ,\qquad 0
\le t < \infty,
\]
so that $ X_{\mfP_{t}(1)}(t) \ge\cdots\ge X_{\mfP_{t}(n)}(t) $.
We may thus write $ Z_{k}(\cdot) = Y_{\mfP_{\cdot}(k)}(\cdot) $
from (\ref{eq: rank proc 1}); loosely speaking, $ \mfP_{t}(k) $ is
the index (name) of the company that occupies the $k$th rank, in terms
of capitalization, at time $ t $.

We shall also introduce the total market capitalization $ X (\cdot)
:= \sum_{i=1}^n X_i (\cdot) $, as well as the market weights (relative
capitalizations) for the individual companies and their ranked
counterparts, respectively,
%
%e4.2 ###
%
\begin{eqnarray} \label{RMW}
\mu_i (t) &:=& { \frac{X_i (t)}{X(t)}} ,\qquad i = 1,\ldots,
n,\quad\mbox{and} \nonumber\\[-8pt]\\[-8pt]
\mu_{(k)} (t) &:=& {\frac{e^{ Z_{k}(t)}}{ X(t)}},\qquad k = 1,\ldots, n
.\nonumber
\end{eqnarray}
\begin{cor} \label{cor: from prop 1}
Under (\ref{eq: drift cond 1}), (\ref{eq: vol cond 1}) and (\ref
{eq: drift cond 2}), the process of ranked deviations $ \widetilde
Z(\cdot):= (Z_{1} (\cdot)- \overline Y(\cdot) , \ldots,
Z_{n}(\cdot) - \overline Y(\cdot) )^{\prime} $ of the
log-capitalizations $ Y_1 (\cdot),\ldots, Y_n(\cdot) $ from their
average, is stable in distribution by Theorem \ref{prop: 1}, and so is
the $ ( (\R_{+})^{n-1} \times\Sigma_{n})$-valued process $
(\Xi(\cdot) , \mfP_{\cdot} ) $, where $ \Xi(\cdot):=(
Z_{1}(\cdot) - Z_{2}(\cdot), \ldots, Z_{n-1}(\cdot) - Z_{n}(\cdot)
)^{\prime} $ is the \textup{rank-gap} process of $ Y(\cdot) $.
\end{cor}

In fact, since $ \widetilde Z(\cdot) $ is obtained by permuting the
components of $ \widetilde Y(\cdot) $, the stability in
distribution of $ \widetilde Y(\cdot) $ implies stability in
distribution for $ \widetilde Z(\cdot) $ from Theorem \ref{prop:
1}. Moreover, the components of the rank-gap process $ \Xi(\cdot) $
can be written as linear combinations of those of $ \widetilde Z(\cdot
) $, and the index process $ \mfP_{\cdot} $ can be seen as $
\mfP_{\cdot} = \mfp^{\widetilde Z(\cdot)} $, where the range $
\Sigma_{n} $ of the mapping $ \mfp$ is a finite set. Thus, the
process $ (\Xi(\cdot) , \mfP_{\cdot} ) $ is stable in distribution.

We shall denote by $ \Lambda^{k,j}(t):= \Lambda_{Z_{k} - Z_{j}}(t)
$ the local time accumulated at the origin by the nonnegative
semimartingale $ Z_{k}(\cdot) - Z_{j}(\cdot) $ up to time $ t
$ for $1 \le k < j \le n $, and set $ \Lambda^{0, 1} (\cdot) \equiv
0 \equiv\Lambda^{n, n+1}(\cdot) $. Then from Theorem 2.5 of Banner
and Ghomrasni \cite{BG07} it can be shown that we have for $ k = 1,
\ldots, n $, $ 0 \le t < \infty$ the dynamics
%
%e4.3 ###
%
\begin{eqnarray} \label{eq: rank proc 2}
d Z_{k}(t) &=& \sum_{i=1}^{n} {\mathbf1}_{Q_{k}^{(i)}}(Y(t)
) \,d Y_{i}(t) \nonumber\\[-8pt]\\[-8pt]
&&{} + ( N_{k}(t) )^{-1} \Biggl[ \sum_{\ell=k+1}^{n} d
\Lambda^{k,\ell}(t) - \sum_{\ell=1}^{k-1} d \Lambda^{\ell,k}(t)
\Biggr].\nonumber
\end{eqnarray}
Here $ N_{k}( t) $ is the cardinality of the set of indices of
those random variables among $ Y_1 (t),\ldots, Y_n (t) $ which
have the same value as $ Z_{k}( t) $, that is, $ N_{k}( t) :=
\lvert\{ i \dvtx Y_{i}( t) = Z_{k}( t) \} \rvert$. Note that
under the assumptions on the coefficients, the finite variation part of
the continuous semimartingale $ Y(\cdot) $ in (\ref{eq: model}) is
absolutely continuous with respect to Lebesgue measure a.s., and it
follows from an application of Fubini's theorem and an estimate of
Krylov \cite{K71} that the Lebesgue measure of the set $ \{t \dvtx
Y_{i}(t) = Y_{j}(t)\} $ is zero a.s. for $ 1\le i\neq j \le n $.
Thus we can verify the sufficient conditions (2.11 and 2.12) of Theorem 2.5
in \cite{BG07}.

Each local time $ \Lambda^{k,\ell}(\cdot) $ is flat away from the
set $ \{ 0 \le t < \infty\lvert Z_{k}(t) = \cdots= Z_{\ell
}(t)\} $; it increases only when the corresponding coordinate
processes collide with each other. Examples in \cite{BP87,IK09}
study such multiple collisions of order three or higher and use
comparisons with Bessel processes in a crucial manner. Here again, the
nonnegative semimartingale $ Z_{k}(\cdot) - Z_{\ell}(\cdot) $ is
compared to an appropriate Bessel process. Since a Bessel process with
dimension $ \delta>1 $ does not accumulate any local time at the
origin (a consequence of Proposition XI.1.11 of \cite{RY99} and of
Theorem~V.48.6 in \cite{RW00}), appropriate comparison arguments yield
the following result; its proof
is in Section \ref{app: 7.2}.
\begin{lm} \label{lm: for prop 2}
Under (\ref{eq: vol cond 1}), the local times $ \Lambda^{k,\ell
}(\cdot) $ generated by triple or higher-order collisions are
identically equal to zero, that is, $ \Lambda^{k, \ell} (\cdot)
\equiv0 $ for $ 1 \le k, \ell\le n $ and $ |k - \ell| \ge2 $,
and (\ref{eq: rank proc 2}) takes for $ k = 1, \ldots, n ,0
\le t < \infty$ the form
%
%e4.4 ###
%
\begin{equation} \label{eq: rank proc 3}
d Z_{k}(t) = \sum_{i=1}^{n} {\mathbf1}_{Q_{k}^{(i)}}(Y(t)) \,d
Y_{i}(t) + \frac{1}{2}\bigl( d \Lambda^{k,k+1}(t) - d \Lambda^{k-1,
k}(t) \bigr).
\end{equation}
\end{lm}
\begin{prop} \label{prop: long gr}
Under the convention (\ref{eq: drift cond 1}) and the assumptions
(\ref{eq: vol cond 1}) and (\ref{eq: drift cond 2}), we obtain a
Strong Law of Large Numbers for local times
%
%e4.5 ###
%
\begin{equation} \label{eq: limit of Lambdak}\quad
\lim_{T\to\infty} \frac{1}{T} \Lambda^{k,k+1}(T) = - 2 \sum_{\ell=1}^k
\Biggl( g_\ell+ \sum_{i=1}^n \gamma_i\theta_{\ell, i} \Biggr)
,\qquad k = 1, \ldots, n-1 ,
\end{equation}
almost surely. Moreover, we obtain the following long-term growth
relations, in addition to those of (\ref{eq: average Y LLN}):
\textup{all log-capitalizations grow at the same rate}
%
%e4.6 ###
%
\begin{equation} \label{eq: Y LLN}
\lim_{T\to\infty} \frac{Y_{i}(T)}{T} = \lim_{T\to\infty}\frac{\log
X_{i}(T)}{T} = \gamma,\qquad i = 1, \ldots, n,
\end{equation}
almost surely. This holds also for the total market capitalization
%
%e4.7 ###
%
\begin{equation} \label{eq: cal X LLN}
\lim_{T\to\infty} \frac{1}{T} \log X(T) = \lim_{T\to\infty} \frac
{1}{T} \log\Biggl( \sum_{i=1}^{n} X_{i}(T) \Biggr) = \gamma,\qquad
\mbox{a.s.,}
\end{equation}
and thus the model is \textup{coherent}; that is, in the notation of
(\ref{RMW}) we have
%
%e4.8 ###
%
\begin{equation}\label{eq: coherence}
\lim_{T\to\infty} \frac{1}{T} \log\mu_{i}(T) = 0 ,\qquad
\mbox{a.s.;\ }i = 1, \ldots, n .
\end{equation}
\end{prop}
\begin{pf}
It follows from Corollary \ref{cor: from prop 1} that
\[
\lim_{T\to\infty} \frac{1}{T} \bigl( Z_{k}(T) - Z_{k+1}(T) \bigr) = 0 ,
\qquad\mbox{a.s.;\ }k = 1, \ldots, n-1 .
\]
Combining this with (\ref{eq: model}), (\ref{eq: occup time 1}) and
(\ref{eq: rank proc 3}), we observe
\begin{eqnarray*}
&& \lim_{T\to\infty} \frac{1}{2T} \bigl( \Lambda^{k-1, k}(T) + \Lambda
^{k+1, k+2}(T) - 2 \Lambda^{k, k+1}(T) \bigr) \\
&&\qquad = g_{k} + \sum_{i=1}^{n} \gamma_{i} \theta_{k, i} - \Biggl(
g_{k+1}+ \sum_{i=1}^{n} \gamma_{i} \theta_{k+1, i}\Biggr) = \mfg_{k} -
\mfg_{k+1} ,\qquad \mbox{a.s.,}
\end{eqnarray*}
where we have set $ \mfg_{k} := g_{k} + \sum_{i=1}^{n} \gamma
_{i}\theta_{k, i} $ for $ k = 1, \ldots, n-1 $. Adding up these
equations over $ k = \ell, \ldots, n-1 $ yields
%
%e4.9 ###
%
\begin{equation}\label{eq: three Lambda}\quad
\lim_{T\to\infty} \frac{1}{2T}\bigl( \Lambda^{\ell-1, \ell}(T) - \Lambda
^{\ell, \ell+1} (T) - \Lambda^{n-1, n}(T) \bigr) = \mfg_{\ell} - \mfg
_{n} ,\qquad \mbox{a.s.}
\end{equation}
for each $ \ell= 1, \ldots, n $; adding up over all these values
of $ \ell$ and using the convention (\ref{eq: drift cond 1}) for
clarity, we obtain
%
%e4.10 ###
%
\begin{equation}
\label{eq: growth last Lambda}
\lim_{T\to\infty} \frac{1}{2T} \Lambda^{n-1, n}(T) = \mfg_{n}
, \qquad \mbox{a.s.}
\end{equation}
In conjunction with (\ref{eq: three Lambda}), we obtain from (\ref{eq:
growth last Lambda}) that for $ k = 1, \ldots, n $
%
%e4.11 ###
%
\begin{equation} \label{eq: growth diff Lambda}\qquad
\lim_{T\to\infty} \frac{1}{2T} \bigl( \Lambda^{k-1, k}(T) - \Lambda
^{k, k+1}(T) \bigr) = \mfg_{k} = g_{k} + \sum_{i=1}^{n} \gamma_{i} \theta
_{k, i} ,\qquad \mbox{a.s.}
\end{equation}
Since $ \sum_{k=1}^n \mfg_k
= 0 $ from (\ref{eq: drift cond 1}) and Corollary \ref{cor: aot 1}, we
obtain (\ref{eq: limit of Lambdak}) from (\ref{eq: growth last
Lambda}) and (\ref{eq: growth diff Lambda}).
From this, (\ref{eq: rank proc 3}), and the strong law of large numbers
for Brownian motion, we get the long-term average growth rate of ranked
log-capitalizations,
\[
\lim_{T\to\infty} \frac{Z_{k}(T)}{T} = \gamma,\qquad \mbox
{a.s.; }k = 1, \ldots, n .
\]
This yields (\ref{eq: Y LLN}); the elementary inequality $ \exp\{
y_{{\mfp^{y}}(1)} \} \le\sum_{i=1}^{n} \exp\{y_{i}\} \le n\times \exp\{
y_{{\mfp^{y}}(1)} \} $ for $ y \in\R^{n} $ then implies (\ref
{eq: cal X LLN}), and equation (\ref{eq: coherence}) is a direct
consequence of (\ref{eq: Y LLN}) and (\ref{eq: cal X LLN}).
\end{pf}
\begin{cor} \label{cor: from prop 2} Under (\ref{eq: drift cond 1}),
(\ref{eq: vol cond 1}) and (\ref{eq: drift cond 2}), the long-term
average occupation times $ \theta_{k, i} $ of (\ref{eq: occup time
1}) satisfy the equilibrium identity
%
%e4.12 ###
%
\begin{equation} \label{eq: occup ident}
\sum_{k=1}^{n} \theta_{k, i} g_{k} + \gamma_{i} = 0,\qquad
i = 1, \ldots, n .
\end{equation}
\end{cor}

Indeed, by substituting (\ref{eq: Y LLN}) into (\ref{eq: model}), and
using the strong law of large numbers for Brownian motion, we obtain the
a.s. identities
\[
\lim_{T \to\infty} \frac{1}{T} \sum_{k=1}^{n} g_{k}\int^{T}_{0}
{\mathbf1}_{Q_{k}^{(i)}} (Y(t))\, d t = - \gamma_{i} ,\qquad
i = 1, \ldots, n,
\]
and so in conjunction with (\ref{eq: occup time 1}) we deduce (\ref{eq:
occup ident}).
\begin{example} \label{ex: 1}
Suppose that the rank-based growth parameters are given as $ g_{n} =
(n-1)g $, $ g_{1} = \cdots= g_{n-1}= -g < 0 $ for some $ g >
0 $. This is the ``Atlas configuration,'' in which the company at the
lowest capitalization rank provides all the growth (or support, as with
the Titan of mythical lore) for the entire structure. Suppose also that
the name-based growth rates $\gamma_1,\ldots, \gamma_n $ satisfy
$ \sum_{i=1}^{n} \gamma_{i} = 0 $ and $ \max_{1\le i \le n} \gamma
_{i} < g $.

It is then checked easily that conditions (\ref{eq: drift cond 1}) and
(\ref{eq: drift cond 2}) are satisfied. By Corollary \ref{cor: aot 1},
the average occupation times $ \{ \theta_{k, i}\} $ exist a.s. We
shall provide an explicit expression for the $ \theta_{k, i} $
under an additional condition (\ref{eq: vol cond 2}) on the correlation
structure, in Section \ref{sec: AOT}. For the time being, let us just
remark that in this case we get directly from (\ref{eq: occup ident})
the long-term proportions of time
\[
\theta_{n, i} = \frac{1}{n} \biggl( 1 - \frac{ \gamma_{i} }{ g }
\biggr),\qquad i = 1, \ldots, n,
\]
with which the various companies occupy the lowest (``Atlas'') rank.
\end{example}

%%%%%%%%%%%%%%%%%%%%%%%%%%%%%%%%%%%%%%%%%%%%%%%%%%%%%%%%%%%%%%%%%%%%%%%%%%%%%%%%%%%%%%%%%
%s5 ###
\section{Invariant measure} \label{sec: Inv meas}
%%%%%%%%%%%%%%%%%%%%%%%%%%%

%s5.1 ###
\subsection{Reflected Brownian motions}
Observe now from (\ref{eq: rank proc 3}) the following representation for the
vector $ \Xi(\cdot) =(\Xi_{1}(\cdot), \ldots, \Xi_{n-1}(\cdot
))^{\prime} $ of gaps in the ranked log-capitalizations $ \Xi
_{k}(\cdot) := Z_{k}(\cdot) - Z_{k+1}(\cdot) = \log(X_{(k)}(\cdot) /
X_{(k+1)}(\cdot)) \ge0 $, $ k=1,\ldots, n-1 ,$
%
%e5.1 ###
%
\begin{equation} \label{eq: RBM form}
\Xi(t) = \Xi(0) + \zeta(t) + \mfR\Lambda(t),\qquad 0 \le t < \infty.
\end{equation}
Here we have set $ \zeta(\cdot):= (\zeta_{1}(\cdot), \ldots, \zeta
_{n-1}(\cdot))^{\prime} $ with
\[
\zeta_{k}(\cdot) = \sum_{i=1}^{n} \int^{\cdot}_{0}
\mathbf{1}_{Q_{k}^{(i)}}(Y(s)) \,d Y(s) - \sum_{i=1}^{n} \int^{\cdot}_{0}
{\mathbf1}_{Q_{k+1}^{(i)}}(Y(s)) \,d Y(s) ;
\]
and we have introduced the vector $ \Lambda(\cdot) := (\Lambda
^{1,2}(\cdot), \ldots, \Lambda^{n-1, n}(\cdot) )^{\prime} = (\Lambda
_{\Xi_{1}}(\cdot),\break\ldots,\Lambda_{\Xi_{n-1}} (\cdot))^{\prime
} $ of local times, as well as the $ ((n-1)\times(n-1)) $ matrix
%
%e5.2 ###
%
\begin{equation} \label{eq: ref matrix}
\mfR:=
\pmatrix{
1 & -1/2 & & &\cr
-1/2 & 1 & -1/2 & &\cr
& \ddots& \ddots& \ddots& \cr
& & -1/2& 1 & -1/2 \cr
& & & -1/2 & 1
}.
\end{equation}
This rank-gap process $ \Xi(\cdot) $ in (\ref{eq: RBM form})
belongs to a class of processes which Harrison and Williams
\cite{HW87Stoch,HW87AP}, Williams \cite{W87b}
and Dai and Williams \cite{DW95} call ``semimartingale reflected (or
regulated) Brownian motions'' (SRBM) in polyhedral domains.

The process $ \Xi(\cdot) $ has state-space $ (\R_{+})^{n-1} $
and behaves like the $ (n-1)$-dimensional continuous semimartingale
$ \zeta(\cdot) $ on the interior of $ (\R_{+})^{n-1} $. When
the face $ \mfF_{k} := \{ (z_{1}, \ldots, z_{n-1})^{\prime} \in(\R
_{+})^{n-1} \vert z_{k} = 0 \} $, $ k = 1, \ldots, n-1 $,
of the boundary is hit, the $k$th component of $ \Lambda(\cdot) $
increases, which causes an instantaneous displacement (reflection) in a
continuous fashion. The directions of this reflection are given by the
entries\vspace*{1pt} in $ \mfr_{k} $, the $k$th column of the matrix $ \mfR
$. For every principal submatrix $ \widetilde\mfR$ of $ \mfR
$, there exists a nonzero vector $ y $ such that $ \widetilde
\mfR y > 0 $, and so the reflection matrix $ \mfR$ satisfies
the so-called \textit{completely-$\mathcal S $} (or ``strictly
semi-monotone'') (see Dai and Williams \cite{DW95} for details) condition
for $ \mathcal S = (\R_{+})^{n-1} $.

Let us define the differential operators $ \mathcal A $ and $
\mathcal D_{k} $, acting on $ C^{2}((\R_{+})^{n-1}) $ functions
%
%e5.4 ###
%e5.3 ###
%
\begin{eqnarray} \label{eq: Diff Oper 1}
[ \mathcal A f ] (z, {\mathbf p}) &:=& \frac{1}{2} \sum_{k,\ell
=1}^{n-1} a_{k, \ell}({\mathbf p}) \frac{\partial^{2} f(z)}{\partial
z_{k}\,
\partial z_{\ell}} + \sum_{k=1}^{n-1} b_{k}({\mathbf p}) \frac{\partial
f}{\partial z_{k}} (z) ,\nonumber\\
&&\eqntext{(z, {\mathbf p}) \in
(\R_{+})^{n-1} \times\Sigma_n ,} \\[-18pt]
\\
{}[ \mathcal D_{k} f] (z) &:=& \langle\mfr_{k} , \nabla f(z)
\rangle,\qquad z \in\mfF_k , k = 1, \ldots,
n-1 .\nonumber
\end{eqnarray}
Here $ ( a_{k, \ell}(\cdot))_{1\le k, \ell\le n-1} $ is the
covariance matrix corresponding to the semimartingale $ \zeta(\cdot
) $ with entries
%
%e5.5 ###
%
\begin{eqnarray} \label{eq: cov for Xi}
a_{k, \ell}({\mathbf p}) &:=& ( \sigma_{k}^{2} + \sigma_{k+1}^{2}
)\cdot{\mathbf1}_{\{k = \ell\}} - \sigma_{k}^{2} \cdot{\mathbf1}_{\{ k =
\ell+1\}} -\sigma_{k+1}^{2} \cdot{\mathbf1}_{\{k = \ell-1\}} \nonumber\\
&&{} + \sum_{m=1}^{n} \bigl( \rho_{{\mathbf p}(k),m} - \rho_{{\mathbf p}(k+1),
m}\bigr)\bigl(\rho_{{\mathbf p}(\ell), m} - \rho_{{\mathbf p}(\ell+1), m}\bigr)
\nonumber\\[-8pt]\\[-8pt]
&&{} + \sum_{(\alpha, \beta) \in\{ (k,\ell), (\ell,
k)\}} \bigl\{ \sigma_{\alpha} \bigl( \rho_{{\mathbf p}(\beta),
\alpha} - \rho_{{\mathbf p}(\beta+1), \alpha} \bigr)\nonumber\\
&&\hspace*{83.2pt}{} + \sigma_{\alpha+1} \bigl(
\rho
_{{\mathbf p}(\beta+1), \alpha+1} - \rho_{{\mathbf p}(\beta), \alpha
+1}\bigr)\bigr\}
\nonumber
\end{eqnarray}
for $ k, \ell= 1, \ldots, n-1 , {\mathbf p} \in\Sigma_{n} $;
whereas the $ ((n-1)\times1)$ vector $\mfr_{k} $ is the $k$th
column of the reflection matrix $ \mfR$. We also define the $
((n-1)\times1) $ drift coefficient vector $ b(\cdot) :=
(b_{1}(\cdot), \ldots, b_{n-1}(\cdot) )^{\prime} $ for the
semimartingale $ \zeta(\cdot) $, with components
%
%e5.6 ###
%
\begin{equation} \label{eq: drift b RBM}
b_{k}({\mathbf p}):= g_{k} + \gamma_{{\mathbf p}^{-1}(k)} - g_{k+1} -
\gamma
_{{\mathbf p}^{-1}(k+1)},\qquad k = 1, \ldots, n-1 , {\mathbf p} \in\Sigma
_{n} .\hspace*{-28pt}
\end{equation}

From Corollary \ref{cor: from prop 1} we know that there exists an
invariant measure $ \nu(\cdot, \cdot) $ for the $ ( (\R
_{+})^{n-1} \times\Sigma_{n})$-valued process $ (\Xi(\cdot), \mfP
_{\cdot}) $. Let us denote by $ \nu_{0}(\cdot) $ the marginal
invariant distribution of $ \Xi(\cdot) $. As a consequence of It\^
o's rule and the formulation of the \textit{submartingale problem} studied
by Stroock and Varadhan \cite{SV71} and Harrison and Williams
\cite{HW87Stoch}, we obtain a characterization of the invariant distribution
$ \nu(\cdot, \cdot) $ for $ ( \Xi(\cdot), \mfP_{\cdot}) $.
\begin{lm} \label{lm: BAR char}
Recall convention (\ref{eq: drift cond 1}), and
conditions (\ref{eq: vol cond 1}) and (\ref{eq: drift cond 2}).
For each $k = 1, \ldots, n-1 $ there is a finite measure $ \nu
_{0k}(\cdot) $, absolutely continuous with respect to Lebesgue
measure on the $k$th face $ \mfF_{k} $, such that the so-called
\textup{Basic Adjoint Relationship} (\textup{BAR}) holds for any $
C^{2}_{b}$-function $ f \dvtx (\R_{+})^{n-1} \rightarrow\R$, namely
%
%e5.7 ###
%
\begin{equation} \label{eq: BAR}\quad
\int_{(\R_{+})^{n-1}\times\Sigma_{n}} [ \mathcal A f ] (z,
{\mathbf p}) \,d \nu(z, {\mathbf p}) + \frac{1}{2} \sum_{k=1}^{n-1} \int
_{\mfF
_{k}} [ \mathcal D_{k} f ] (z) \,d \nu_{0k}(z) = 0 .
\end{equation}
\end{lm}

This condition is necessary for the stationarity of $ \nu(\cdot, \cdot
) $. A proof of Lemma \ref{lm: BAR char} is given in Section \ref
{sec: prf BAR char}. It is not easy to solve (\ref{eq: BAR}) in
general; however, following Harrison and Williams \cite{HW87AP}, we may
obtain an explicit formula for the invariant joint distribution $ \nu
(\cdot, \cdot) $ under the so-called {\sl skew symmetry condition}
between the covariance and reflection matrices (see Theorem \ref{prop:
linear growing var} and Corollaries \ref{cor: aot 2} and \ref{cor:
market share}).
\begin{lm} \label{lm: vol cond 2} Assume that the rank-based variances
$ \{\sigma_{k}^{2}\} $ grow linearly, and that there are \textup{no
name-based correlations} in (\ref{eq: model}), that is,
%
%e5.8 ###
%
\begin{eqnarray} \label{eq: vol cond 2}
\sigma_{2}^{2} - \sigma_{1}^{2} &=& \sigma_{3}^{2} - \sigma_{2}^{2} =
\cdots=\sigma_{n}^{2} - \sigma_{n-1}^{2},\nonumber\\[-8pt]\\[-8pt]
\rho_{i, j} &=& 0 ,\qquad 1\le i , j \le n .\nonumber
\end{eqnarray}
Then the components of the covariance matrix $ \mfA\equiv(
\mathfrak{ a}_{k, \ell} )_{1\le k , \ell\le n-1} $ from (\ref
{eq: cov for Xi}) become
\[
\mathfrak{ a}_{k, \ell}
= ( \sigma_{k}^{2} + \sigma_{k+1}^{2})\cdot{\mathbf1}_{\{k = \ell
\}} - \sigma_{k}^{2} \cdot{\mathbf1}_{\{ k = \ell+1\}} -\sigma_{k+1}^{2}
\cdot{\mathbf1}_{\{k = \ell-1\}}
\]
and do \textup{not} depend on the permutation $ {\mathbf p} \in\Sigma_{n}$.
Moreover, the matrix $ \mathfrak A $ satisfies the so-called
\textup{skew symmetry} condition,
%
%e5.9 ###
%
\begin{equation}\label{eq: SSC}
( 2 \mfD- \mathfrak{ H} \mfD- \mfD\mathfrak{ H} - 2 \mfA
)_{k, \ell} = 0 ,\qquad 1 \le k, \ell\le n-1 .
\end{equation}
Here we have introduced the diagonal matrix $ \mfD:= \operatorname
{diag}(\mfA) $, and the $ ((n-1)\times(n-1)) $ matrix $
\mathfrak{ H} := I - \mfR$ from the reflection matrix $ \mfR$
in (\ref{eq: ref matrix}).
\end{lm}

Lemma \ref{lm: vol cond 2} is proved by straightforward computation;
details are in Section 5.5 of~\cite{IK09}. Note that, even under (\ref
{eq: vol cond 2}), the operator (\ref{eq: Diff Oper 1}) still depends
on the permutation $ {\mathbf p} $ through the drift component $
b({\mathbf p}) $ for $ {\mathbf p} \in\Sigma_{n} $ in (\ref{eq: drift b RBM}).
\begin{theorem} \label{prop: linear growing var}
Under (\ref{eq: drift cond 1}), (\ref{eq: vol cond 1}), (\ref{eq:
drift cond 2}) and (\ref{eq: vol cond 2}), the invariant joint
distribution $ \nu(\cdot, \cdot) $ of the $ ( (\R_{+})^{n-1}
\times\Sigma_{n})$-valued process $ (\Xi(\cdot), \mfP_{\cdot})
$ is
%
%e5.10 ###
%
\begin{equation} \label{eq: joint dist}
\nu(A\times B) := \Biggl( \sum_{{\mathbf q} \in\Sigma_{n}} \prod
_{k=1}^{n-1} \lambda_{{\mathbf q} , k}^{-1}\Biggr)^{-1} \sum_{{\mathbf p} \in
B} \int_{A}\exp( - \langle\lambda_{{\mathbf p}} , z \rangle)
\,d z
\end{equation}
for any measurable sets $ A
\subset(\R_{+})^{n-1} $ and $ B \subset\Sigma_{n} $, where $
\lambda_{{\mathbf p}} := (\lambda_{{\mathbf p}, 1}, \ldots,\break \lambda
_{{\mathbf p},
n-1})^{\prime} $ is the vector with components
%
%e5.11 ###
%
\begin{equation} \label{eq: lambda}
\lambda_{{\mathbf p}, k} := \frac{- 4 \sum_{\ell=1}^{k} (g_{\ell} +
\gamma_{{\mathbf p}(\ell)})}{\sigma_{k}^{2} + \sigma_{k+1}^{2}},\qquad
{\mathbf p} \in\Sigma_{n} , 1\le k \le n-1 .
\end{equation}
In particular, the density $ \wp(\cdot) $ of the marginal
invariant distribution $ \nu_{0}(\cdot) $ of $ \Xi(\cdot) $ has
the \textup{sum-of-products-of-exponentials} form
%
%e5.12 ###
%
\begin{equation} \label{eq: prod density form 1}\quad
\wp(z) := \Biggl( \sum_{{\mathbf q} \in\Sigma_{n}} \prod_{k=1}^{n-1}
\lambda_{{\mathbf q} , k}^{-1}\Biggr)^{-1} \sum_{{\mathbf p} \in\Sigma_{n}}
\exp
( - \langle\lambda_{{\mathbf p}} , z \rangle),\qquad z \in
(\R_{+})^{n-1} .
\end{equation}
\end{theorem}
\begin{pf}
First, we carry out a linear transformation of the state space to
remove the correlation between the components of $ \Xi(\cdot) $;
this is possible, because the covariance matrix $ \mfA$ does
\textit{not} depend on the index process $ \mfP_{\cdot} $, under (\ref{eq:
vol cond 2}) from Lemma \ref{lm: vol cond 2}.
Let $ \mfU$ be the matrix whose columns are the orthogonal
eigenvectors of the covariance $ \mfA$, and let $ \mfL$ be
the corresponding diagonal matrix of eigenvalues such that $\mfL=
\mfU^{\prime} \mfA\mfU$. Define $ \widetilde\Xi(\cdot):= \mfL
^{-1/2} \mfU\Xi(\cdot) $. By this deterministic rotation and
scaling, we obtain
%
%e5.13 ###
%
\begin{equation} \label{eq: def tilde Xi}
\widetilde\Xi(t) = \widetilde\Xi(0) + \widetilde\zeta(t) +
\widetilde\mfR\Lambda(t) ,\qquad 0 \le t < \infty,
\end{equation}
from (\ref{eq: RBM form}) where $ \widetilde\zeta(\cdot) = \mfL
^{-1/2} \mfU\zeta(\cdot) $ is a Brownian motion with drift
coefficient $ \widetilde b( \cdot):= \mfL^{-1/2} \mfU b( \cdot)
$ and $ b(\cdot) $ is defined in (\ref{eq: drift b RBM}).
We may regard $ \widetilde\Xi(\cdot) $ as a reflected Brownian
motion in a new state space $ \mfS:= \mfL^{-1/2} \mfU(\R
_{+})^{n-1} $ with faces $ \widetilde\mfF_{k} := \mfL^{-1/2} \mfU
\mfF_{k} $, $ k = 1, \ldots, n-1 $. The transformed reflection
matrix $ \widetilde\mfR:= \mfL^{-1/2} \mfU\mfR$ can be
written\vspace*{1pt} $ \widetilde\mfR= (\widetilde\mfN+ \widetilde\mfQ)
\widetilde\mfC= (\widetilde\mfr_{1}, \ldots, \widetilde\mfr
_{n-1}) $, where $ \widetilde\mfC:= \mfD^{-1/2} $, $ \mfD:=
\operatorname{diag}(\mfA) $, $ \widetilde\mfN:= \mfL^{1/2} \mfU
\widetilde\mfC= (\widetilde\mfn_{1}, \ldots, \widetilde\mfn_{n-1}
) $, $ \widetilde\mfQ:= \mfL^{-1/2} \mfU\mfR\widetilde
\mfC^{-1} - \widetilde\mfN= (\widetilde\mfq_{1}, \ldots, \widetilde
\mfq_{n-1}) $. The constant vectors $ \widetilde\mfr_{k},
\widetilde\mfq_{k}, \widetilde\mfn_{k} $, $ k = 1, \ldots, n-1
$, are $ ((n-1)\times1) $ column vectors.

The corresponding differential operators $ \widetilde{\mathcal A} ,
\widetilde{\mathcal D_{k}} $ and their adjoints $ \widetilde{
\mathcal A}^{\ast} , \widetilde{\mathcal D}_{k}^{\ast} $ are defined by
%
%e5.14 ###
%
\begin{eqnarray} \label{eq: Diff Oper 2}
[ \widetilde{\mathcal A} f ] (z, {\mathbf p}) &:=& \tfrac{1}{2}
\Delta f(z)+ \langle\widetilde b({\mathbf p}) , \nabla f(z) \rangle,\nonumber\\
{}[ \widetilde{\mathcal D}_{k} f] (z) &:=& \langle\widetilde
\mfr_{k} , \nabla f(z) \rangle,  \nonumber\\[-8pt]\\[-8pt]
{}[ \widetilde{\mathcal A}^{\ast} f ] (z, {\mathbf p}) &:=& \tfrac
{1}{2} \Delta f(z)- \langle\widetilde b({\mathbf p}) , \nabla f(z)
\rangle,\nonumber\\
{}[ \widetilde{\mathcal D}^{\ast}_{k} f] (z) &:=& \langle\widetilde
{\mfr}^{\ast}_{k} , \nabla f(z) \rangle,
\nonumber
\end{eqnarray}
where we define the adjoint direction
$ \widetilde{\mfr}_{k}^{\ast} := \widetilde\mfn_{k} - \widetilde\mfq
_{k} + \langle\widetilde\mfn_{k}, \widetilde\mfq_{k}\rangle
\widetilde\mfn_{k} $
of reflection to $ \widetilde{\mfr}_{k} $
for $ k = 1, \ldots, n-1 $, $ z \in(\R_{+})^{n-1} $, $ {\mathbf p} \in
\Sigma_{n} $.

With these differential operators as in Lemma \ref{lm: BAR char}, we
obtain the (BAR) for the process $ ( \widetilde\Xi(\cdot), \mfP
_{\cdot} ) $ and its invariant distribution $ \widetilde\nu(\cdot
, \cdot) $; that is, for every $ k = 1, \ldots, n-1 $, there
exists a finite measure $ \{\widetilde\nu_{0k}(\cdot)\} $ which is
absolutely continuous with respect to the $ (n-2)$-dimensional
Lebesgue measure on $ \widetilde\mfF_{k} $ and such that for any
$ C^{2}_{b}$-function $ f \dvtx\mfS\mapsto\R$ we have
%
%e5.15 ###
%
\begin{equation} \label{eq: BAR 2}
\int_{\mfS\times\Sigma_{n}} [ \widetilde{\mathcal A} f] (z,
{\mathbf p}) \,d \widetilde\nu(z, {\mathbf p}) + \frac{1}{2} \sum
_{k=1}^{n-1} \int_{\widetilde\mfF_{k}} [ \widetilde{ \mathcal
D}_{k} f] (z) \,d \widetilde\nu_{0k}(z) = 0 .
\end{equation}
Our argument, especially from here onward, relies heavily on the
elaborate analysis given by Harrison and Williams \cite{HW87Stoch,HW87AP}.
The main distinction between their setting and ours is
in the drift coefficient $ b( \cdot) $, which here varies from
chamber to chamber as well as within each chamber, and is evaluated
along the path of the index process $ \mfP_{\cdot} $. Here,
however, we can use the following observation.
\begin{lm} \label{lm: BAR/SSC}
The following two conditions are equivalent:
\begin{enumerate}[(ii)]
\item[(i)] For each\vspace*{2pt} collection of constants $ \{ g_{k}, \gamma
_{i} ; 1 \le i, k \le n \} $, there are $
(n-1)$-dimensional vectors $ \widetilde\lambda_{{\mathbf
p}}:=(\widetilde\lambda_{{\mathbf p},1}, \ldots, \widetilde\lambda
_{{\mathbf p},n-1})^{\prime} $ for $ {\mathbf p} \in\Sigma_{n} $, such
that a
probability measure in the form of sum of products of exponentials
%
%e5.16 ###
%
\begin{equation} \label{eq: prod dens form 2}
\widetilde\nu(A \times B) := c \sum_{\mathbf{ p} \in B } \int_{A} \exp
( \langle\widetilde\lambda_{{\mathbf p}}, z\rangle) \,d z =:
\sum_{{\mathbf p} \in B} \int_{A} \widetilde\wp_{{\mathbf p}}(z) \,d z
\end{equation}
for measurable sets $ A \subset\mfS$ and $ B \subset\Sigma
_{n} $,
satisfies (\ref{eq: BAR 2}) for $ f(\cdot) \in C_{c}^{2}(\mfS) $,
where $ c $ in (\ref{eq: prod dens form 2}) is a normalizing constant.
\item[(ii)] The covariance and the direction of reflection
satisfy the skew symmetry condition (\ref{eq: SSC}).
\end{enumerate}
\end{lm}

Indeed, substituting (\ref{eq: prod dens form 2}) into (\ref{eq: BAR
2}) and combining the summation over $ {\mathbf p} \in\Sigma_{n} $, we
observe that the left-hand side of
(\ref{eq: BAR 2}) becomes
\[
\sum_{{\mathbf p} \in\Sigma_{n}} \Biggl\{ \int_{\mfS} [ \widetilde
{\mathcal A} f ](z, {\mathbf p}) \cdot\widetilde\wp_{{\mathbf p}}(z)\,
d z + \frac{1}{2} \sum_{k=1}^{n-1} \int_{\widetilde\mfF_{k}} [
\widetilde{ \mathcal D}_{k} f ] (z) \cdot\widetilde\wp_{{\mathbf p}}(z)
\,d z \Biggr\}
\]
for $ f \in C_{c}^{2}(\mfS) $, where the expression in the curly
bracket corresponds exactly to the BAR condition studied in
\cite{HW87AP} with some differences in notation. This way, we may
reduce our
problem to the case of \cite{HW87AP}. Following the proof of Lemma 7.1
in \cite{HW87AP}, we observe that condition (i) in Lemma \ref{lm:
BAR/SSC} is equivalent to the following conditions (iii) and (iv), where:

\begin{enumerate}[(iii)]
\item[(iii)] $ [ \widetilde{\mathcal A}^{\ast} \widetilde{\wp
}_{\cdot} ](\cdot, \cdot) = 0 $ in $ \mfS\times\Sigma_{n}
$, and

\item[(iv)] $ [ \widetilde{\mathcal D}^{\ast}_{k}
\widetilde{\wp}_{\mathbf p}](\cdot) = 2 b_{k}(\cdot) \widetilde\wp
_{\mathbf p}(\cdot) $ on $ \widetilde\mfF_{k} $ for $ k = 1,
\ldots, n-1 $, $ {\mathbf p} \in\Sigma_n $.
\end{enumerate}

Here the adjoint operators $ \widetilde{\mathcal A}^{\ast}
$, $ \widetilde{\mathcal D}_{k}^{\ast} $ are defined in (\ref{eq:
Diff Oper 2}).

Then\vspace*{1pt} the same reasoning as in the proof of Theorem 2.1 in \cite{HW87AP}
yields our Lemma \ref{lm: BAR/SSC}, and we obtain $ \widetilde\lambda
_{\mathbf p} = 2 (I - \widetilde\mfN\widetilde\mfQ)^{-1} b({\mathbf p}) $
for $ {\mathbf p} \in\Sigma_{n} $ along the way. This gives the
invariant distribution $ \widetilde\nu(\cdot) $ of $ \widetilde
\Xi(\cdot) $ in (\ref{eq: def tilde Xi}). Now transforming back to
$ \Xi(\cdot) $, we obtain (\ref{eq: lambda}), (\ref{eq: joint
dist}) and then (\ref{eq: prod density form 1}).
\end{pf}
\begin{example} \label{ex: first order}
With $ \gamma_{i} = 0 $, $ \rho_{i, j} = 0 $, $ 1 \le i, j
\le n $ and $ \sigma_{1}^{2} = \cdots= \sigma_{n}^{2} $, we
recover the case studied by Banner, Fernholz and Karatzas \cite{BFK05}
and Pitman and Pal~\cite{PP08}. Our Theorem \ref{prop: linear growing
var} is an extension of their results, to the case of variances that
are not necessarily equal and, as far as
the second of these papers is concerned, to a finite number of particles.
\end{example}

%s5.2 ###
\subsection{Average occupation times} \label{sec: AOT}
The long-term average occupation time $ \theta_{{\mathbf p}} $ of the
vector process $ \mathfrak{ X}(\cdot) $ in the polyhedral chamber
$ \RR_{{\mathbf p}} $ of (\ref{eq: theta p def}) is the probability
mass $ \nu_{1} ({\mathbf p}) := \nu( (\R_{+})^{n-1}, {\mathbf p}
) $ assigned to such a particular chamber by the marginal invariant
distribution of the index process $ \mfP_{\cdot} $, which we can
compute directly from (\ref{eq: joint dist}) for $ {\mathbf p} \in\Sigma
_{n} $.
\begin{cor} \label{cor: aot 2}
Under the assumptions of Theorem \ref{prop: linear growing var}, the
long-term average occupation time $ \theta_{\mathbf p} $ of $
\mathfrak{ X}(\cdot) $ in the chamber $ \RR_{\mathbf p} $ for $
{\mathbf p} \in\Sigma_{n} $, and the long-term proportion $ \theta_{k,
i} $ of time spent by company $i$ in the $k$th rank as in (\ref{eq:
occup time 1}), are explicitly given by the respective formulae
%
%e5.17 ###
%
\begin{equation}\label{eq: theta p}
\theta_{{\mathbf p}}
= \Biggl( \sum_{\mathbf{ q} \in\Sigma_{n}} \prod_{j=1}^{n-1} \lambda
_{{\mathbf q} , j}^{-1} \Biggr)^{-1} \prod_{j=1}^{n-1} \lambda_{{\mathbf p},
j}^{-1} \quad\mbox{and}\quad \theta_{k, i} = \sum\theta_{\mathbf p}.
\end{equation}
Here $ \lambda_{\mathbf p} $ is as in (\ref{eq: lambda}), and the
summation for $ \theta_{k, i} $ is taken over the set $ \{ {\mathbf p}
\in\Sigma_{n} \vert{\mathbf p}(k) = i \} $ for $ 1\le i,
k \le n $.
\end{cor}

From Corollary \ref{cor: from prop 2}, the average occupation times $
(\theta_{k, i}) $ satisfy the equilibrium identity (\ref{eq: occup
ident}). As a sanity check, we verify this
identity for the expressions of (\ref{eq: theta p}), through some
algebraic computations in Section \ref{sec: sanity check}.
\begin{example} \label{ex: namvar}
It should be noted that in the presence of name-based variances,
\eqref{eq: theta p} can fail significantly. Consider the case where
$n=3$, with $\gamma_i=0$, for $i=1,2,3$; $\sigma_k=\sigma>0$, for
$k=1,2,3$; $g_3=g>0$, $g_2=0$ and $g_1=-g$; all
$\rho_{i,j}$ is zero for $i, j=1,2,3$ except $\rho_{3,3}=\rho\gg\sigma
$. In this case, $ Y_1 (\cdot) $ and $ Y_2 (\cdot) $ will vibrate
quietly in the middle with variance rate $\sigma^2$, while $ Y_3 (\cdot
) $, with much greater variance rate $ (\sigma+\rho)^{2} $, will
be wandering far and wide. From Corollary \ref{cor: aot 1} and (\ref
{eq: occup ident}) we obtain
%
%e5.18 ###
%
\begin{equation}\label{eq: theta namvar}
\vartheta= ( \theta_{k,i} )_{1\le i,k \le3} = \pmatrix{
\dfrac{1-\alpha}{2} & \dfrac{1-\alpha}{2} & \alpha\vspace*{2pt}\cr
\alpha& \alpha& 1 - 2 \alpha\vspace*{2pt}\cr
\dfrac{1-\alpha}{2} & \dfrac{1-\alpha
}{2} &\alpha},
\end{equation}
where the parameter $ \alpha$ is in the interval $ (1 /
3,1 / 2) $ for $ \rho> 0 $. The upper bound $ 1 / 2
$ is obtained as $ \lim_{\rho\to\infty} \theta_{1,3} $.
Without name-based variances, that is, if the $\rho_{i,j}$ were all
zero, the $ Y_i (\cdot) $ would each spend the same proportion of
time in every rank, yielding a matrix $ \vartheta$ in (\ref{eq:
theta namvar}) with all entries equal to $ 1 / 3 $ from
Corollary \ref{cor: aot 2}. This gives the lower bound $ 1 / 3 $.
\end{example}
\begin{example} \label{ex: occup 1}
Let us consider a numerical computation of $ (\theta_{k, i}) $ for
descending name-based drifts $ \gamma_{i} $ and ascending
rank-based drifts $ g_{k} $, for example, $ n=10 $ and $
\sigma_{k}^{2} = 1+k $, as well as $ g_{k} = -1 $ for $ k= 1,
\ldots, 9 $, $ g_{10} = 9 $, $ \gamma_{i} = 1 - ( 2 i )
/ (n+1) $ for $ i = 1, \ldots, n $. This is a rather
extreme case of Example \ref{ex: 1}, with $g=1$. The overall maximum is
$ \theta_{1,1} = 0.5184 $, and the overall minimum is $ \theta
_{1, 10} = 0.00485 $. The company ``$ i=1 $'' stays at the first
rank longer than any other companies because of its relatively strong
name-based drift; whereas the company ``$ i=10 $'' stays at the first
rank only for a tiny amount of time because of its relatively poor
name-based drift.

%
%f1 ###
%
\begin{figure}

\includegraphics{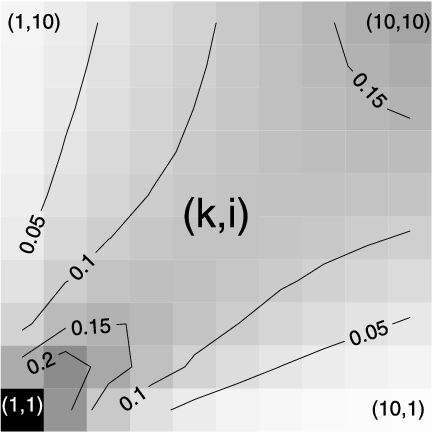}

\caption{Different values of $ \{\theta_{k,i}\} $ for
$ (k, i) $, when the parameters are specified for an extreme case
in Example \protect\ref{ex: occup 1}.} \label{fig: heatmap 1}
\end{figure}

Figure \ref{fig: heatmap 1} shows a gray scale heat map for the
different values of $ \{\theta_{k,i}\} $; of course we know from
Example \ref{ex: 1} that $ \theta_{10, i} = i / 55 $, $ i=1,\ldots, 10 $.
\end{example}

For a larger number of companies, say $ n \sim5000 $, it seems
rather hopeless for the current computational environment to perform
direct computations of $ \theta_{k,i} $ via the sum of (\ref{eq:
theta p}) over $ (n-1)! $ permutations in general.

%s5.3 ###
\subsection{Capital distribution curve}
The capital distribution curve is the log-log plot of market weights
in descending order, as in (\ref{LogLog}). The empirical capital
distribution curves, for the U.S. stock market over the seven decades
1929--1999, are shown in \cite{F02} (Figure 5.1 on page 95). Our next
result computes the capital distribution curves directly from Theorem
\ref{prop: linear growing var}, from the gaps $ \Xi_k(\cdot) = \log(
\mu_{(k)}(\cdot) / \mu_{(k+1)}(\cdot) ) $ in the ranked log-market-weights
to the ranked log-market-weights $ \mfc_k(\cdot)
:= \log\mu_{(k)}(\cdot) $ themselves.
\begin{cor} \label{cor: market share} Under the assumptions of Theorem
\ref{prop: linear growing var}, the ranked market weights $ \mu
_{(1)}(\cdot), \ldots, \mu_{(n)}(\cdot) $ in (\ref{LogLog}), (\ref
{RMW}) have invariant distribution with
%
%e5.19 ###
%
\begin{equation}\label{eq: dens m}\quad
\wp(m_{1}, \ldots, m_{n-1} ) = \sum_{{\mathbf p} \in\Sigma_{n}} \Biggl[
\theta_{\mathbf p} \cdot\prod_{k=1}^{n-1} \lambda_{{\mathbf p}, k} \cdot
\Biggl( \prod_{j=1}^{n} m_{j}^{\lambda_{{\mathbf p},j} - \lambda_{{\mathbf p},j-1}
+1} \Biggr)^{-1} \Biggr]
\end{equation}
as its density, for $ 0 < m_{n} \le m_{n-1} \le\cdots\le m_{1} <
1 $ and $ m_{n} = 1- m_{1} - \cdots- m_{n-1} $. Here we set $
\lambda_{{\mathbf p}, 0} = 0 = \lambda_{{\mathbf p}, n} $, $ {\mathbf
p} \in
\Sigma_{n} $, for notational simplicity.

Moreover, the log-ranked market weights $ \mfc_{k}(\cdot) = \log\mu
_{(k)}(\cdot) $ have invariant distribution with density
%
%e5.20 ###
%
\begin{equation}\label{eq: dens c}\qquad
\wp(c_{1}, \ldots, c_{n-1} ) = \sum_{{\mathbf p} \in\Sigma_{n}} \Biggl[
\theta_{\mathbf p} \cdot\prod_{j=1}^{n-1}\bigl( \lambda_{{\mathbf p}, j}
\cdot
e^{- (\lambda_{{\mathbf p},j} - \lambda_{{\mathbf p}, j+1} ) c_{j}} \bigr)
\cdot
e^{\lambda_{\mathbf{p}, n-1} c_{n}}\Biggr]
\end{equation}
for $ -\infty< c_{n} \le\cdots\le c_{2} \le c_{1} < 0 $, $
c_{n} = \log( 1- \sum_{j=1}^{n-1} e^{c_{j}} ) $.
\end{cor}

From the invariant density functions given by (\ref{eq: prod density
form 1}) and (\ref{eq: dens m}), (\ref{eq: dens c}) [or simply (\ref
{eq: theta p})], the piecewise linear capital distribution curve (\ref
{LogLog}) has the expected slope
%
%e5.21 ###
%
\begin{equation} \label{eq: slope CDC}\quad
\E^{\nu} \biggl[ \frac{\log\mu_{(k+1)} - \log\mu_{(k)}}{\log(k+1) -
\log k} \biggr] = - \frac{\E^{\nu}( \Xi_{k} )}{\log( 1+ k^{-1})}
= - \frac{ \sum_{{\mathbf p} \in\Sigma_{n}} \theta_{{\mathbf p}} \lambda
_{{\mathbf p}, k}^{-1}}{\log( 1+ k^{-1})}
\end{equation}
between the $k$th and the $(k+1)$st ranked stocks for
$ k = 1, \ldots, n-1 $, and the initial value
\[
\E^{\nu} \bigl( \log\mu_{(1)}\bigr) = \E^{\nu} ( \mfc_{1}) = \E^{\nu} \bigl[ -
\log\bigl( 1 + e^{-\Xi_{1}} + e^{-(\Xi_{1} + \Xi_{2})} + \cdots+
e^{-(\Xi_{1} + \cdots+ \Xi_{n-1})} \bigr) \bigr]
\]
for the first rank. From (\ref{eq: joint dist}) this expected initial
value may be obtained through a Monte Carlo simulation of generating
$ (n-1) $ independent exponential random variables with intensities
$ \lambda_{{\mathbf p}, j} $ for $ j = 1, \ldots, n-1 $, $ {\mathbf p}
\in\Sigma_{n} $. From (\ref{eq: slope CDC}) we obtain the
following simple criterion for convexity (or concavity) of the expected
capital distribution curves.
\begin{cor} \label{cor: CDC} Under the assumptions of Theorem \ref
{prop: linear growing var}, a sufficient condition for the expected
capital distribution curve $ \log k \mapsto\E^{\nu}(\log\mu
_{(k)}) $ under the invariant distribution $ \nu$ to be
\textup{convex} (resp., \textup{concave}), is that
%
%e5.22 ###
%
\begin{equation} \label{eq: criterion CDC}
\lambda_{{\mathbf p}, k+1} \log\biggl( 1+ \frac{1}{k+1}\biggr) - \lambda_{{\mathbf
p}, k} \log\biggl( 1+ \frac{1}{k}\biggr) \ge0\qquad \forall{\mathbf p} \in\Sigma_{n}
\end{equation}
(resp., $\le$) hold for each $ k = 1, \ldots, n-2 $, where $
\lambda_{{\mathbf p}, k} $ is given in (\ref{eq: lambda}).
\end{cor}
\begin{example}\label{ex: first order cnt}
Let us consider the first-order Atlas model which is a combination of
the ``Atlas configuration'' in Example \ref{ex: 1} with the further
restrictions of Example \ref{ex: first order}; to wit, $ g_{n} =
(n-1)g $, $ g_{1} = \cdots= g_{n-1} = -g <0 $ for some $ g>0
$, as well as $ \gamma_{i} = 0 $, $ \rho_{i, j} = 0 $, $ 1
\le i, j \le n $, and $ \sigma_{1}^{2} = \cdots= \sigma_{n}^{2} =
\sigma^{2} > 0 $ for some $ \sigma^{2} > 0 $. From Corollary \ref
{cor: CDC}, the expected capital distribution curve is \textit{convex} but
\textit{almost linear} for larger $ k $. Indeed, the quantity $
\lambda_{{\mathbf p}, k} \log(1+ k^{-1})= 2(g k / \sigma^{2}) \cdot
\log(1+ k^{-1}) $ increases in $ k \ge1 $, and converges to
$ 2g / \sigma^{2} $, as $ k \uparrow\infty$, for all $
{\mathbf p} \in\Sigma_{n} $, and so the difference in (\ref{eq:
criterion CDC}) is positive for each $ k = 1, \ldots, n-2 $ but
decreases to zero
quite rapidly in the order of $ O(k^{-2}) $, as $ k \uparrow
\infty$. Another explanation of such linearity (``Pareto line'') of
the capital distribution curves from an application of Poisson point
processes can be found in Example 5.1.1 on page 94 of \cite{F02}.
\end{example}
\begin{example} \label{ex: linear grow var}
Suppose now that we change only the rank-based variances in Example \ref
{ex: first order cnt}; namely, we take linearly growing variances $
\sigma_{k}^{2} = k\sigma^{2} $ for some $ \sigma^{2} > 0 $, $ k
= 1, \ldots, n $. Then
\[
\lambda_{{\mathbf p}, k} \log\biggl(1+\frac{1}{k}\biggr) =
\frac{4kg}{(2k+1)\sigma^{2}} \cdot\log\biggl(1+ \frac{1}{k}\biggr)
\]
is decreasing in $ k \ge1 $ for every $ {\mathbf p} \in\Sigma
_{n} $, and so the difference in (\ref{eq: criterion CDC}) is
negative for each $k = 1, \ldots, n-2 $. Thus, from Corollary \ref
{cor: CDC}, the expected capital distribution curve becomes \textit{concave}.
\end{example}
\begin{example}[(``Pure'' hybrid market conjecture)] \label{ex: pure
hybrid} A pure hybrid market
is one in which all the parameters are determined by the ``name'' of
the stock, with the exception of the growth rate of the smallest
stock. The log-capitalization $ Z_n(\cdot) $ of the smallest stock
has its growth rate incremented by $g>0$, as
in the Atlas model. Hence, this market will look like
\[
d Y_i(t) =
\cases{
-\gamma_i \,d t + \sigma_i \,d W_i(t) , &\quad if $Y_i(t) \ne Z_n
(t)$,\cr
(g-\gamma_i) \,d t + \sigma_i \,d W_i(t) , &\quad if $Y_i(t) = Z_n
(t)$,}
\]
for $i=1,\ldots,n$ and $t\in[0, \infty)$, where $\gamma_i > 0$, $\sigma
_i > 0$, and
$g=\sum_{i=1}^{n}\gamma_i$. We conjecture that the capital distribution
curve for
this market is convex.

This conjecture is based on the following reasoning: The Atlas stock
$ Z_{n}(\cdot) $
performs a role similar to a local time process, reflecting each stock
away from the bottom position. Hence, outside the set where $Y_i(\cdot) =
Z_n(\cdot)$, the distance $Y_i(\cdot)-Z_{n}(\cdot)$ will be
approximately exponentially
distributed. Accordingly, suppose we replace
$Y_i(\cdot)-Z_{n}(\cdot)$ by an exponentially distributed random variable
$ {\mathbf Z}_i$ with rate parameter $ \alpha_i = \sigma^2_i/(2\gamma
_i) $
\[
P\{{\mathbf Z}_i > x\} = e^{-\alpha_i x} ,\qquad x > 0 ,
i = 1, \ldots, n .
\]
Let $ {\mathbf Z} $ represent a generic member of such random variables
$ ({\mathbf Z}_i , i = 1, \ldots, n) $
as a mixed exponential distribution
\[
P\{{\mathbf Z} > x \} = \frac{1}{n}\sum_{i=1}^{n} e^{-\alpha_i x}
,\qquad x > 0 ,
\]
and define $ z_{(k)} $ as $ P\{{\mathbf Z} > z_{(k)}\} = k / n $ for
$ k = 1, \ldots, n $.
In this case, the capital distribution
curve is approximately proportional to the graph of $ z_{(k)} $
versus $\log k$, and this graph, $ \log k \mapsto z_{(k)} $, $ k =
1, \ldots, n $, will be convex on average.
In fact, the graph of
$
\log(k/n) = \log(\sum_{i=1}^{n} e^{-\alpha_i x}/n) ,
$
where $\log k$ is considered to be a function of $x$,
is convex, because with
$
\phi(x) := \sum_{i=1}^{n} e^{-\alpha_i x} ,
$
\[
\frac{d^2}{dx^2}\log k
= \frac{\phi''(x)\phi(x) - (\phi'(x))^2}{(\phi(x))^2}
= \frac{\sum_{i,j=1}^n (\alpha_i - \alpha_j)^2 e^{-(\alpha_i + \alpha
_j)x}}{2(\phi(x))^2} \ge0 .
\]
Note, of course, that this holds for the random variables
$ ({\mathbf Z}_i, i = 1, \ldots, n) $ and~${\mathbf Z}$, but that it
holds for the process $ Y_i(\cdot) $ is only a conjecture. This
conjecture is of interest because, historically, capital distribution
curves appear to be concave which could imply that rank-based
parameters as well as name-based parameters are needed to explain
stock market behavior.
\end{example}
\begin{example}\label{ex: ExpCDC2}
To see different shapes of the expected capital distribution curve
under different parameter configurations apart from Examples \ref{ex:
first order cnt} and \ref{ex: linear grow var}, let us consider a pure
hybrid market whose drift and volatility coefficients do not depend on
ranks, except for the smallest (Atlas) stock. For example, take $ n =
5000 $, $ g_{k} = 0 , 1\le k \le n-1 $, $ g_{n} =
c_{\ast}(2n-1) $, $ \gamma_{1} = - c_{\ast} $, $ \gamma_{i} =
-2c_{\ast} , 2 \le i \le n $, $ \sigma_{k}^{2} = 0.075
$, $ 1 \le k \le n $ and $ \rho_{i, j} = 0 $ for $ 1\le i, j
\le n $ with a parameter $ c_{\ast} = 0.02$. These parameters
satisfy the assumptions of Theorem \ref{prop: linear growing var}. We
cannot apply Corollary \ref{cor: CDC} because the difference in
(\ref{eq: criterion CDC}) is positive on $ \{ {\mathbf p} \in\Sigma_{n}
\dvtx {\mathbf p}(k+1) \neq1 \} $ but negative on its (smaller)
%
%f2 ###
%
\begin{figure}

\includegraphics{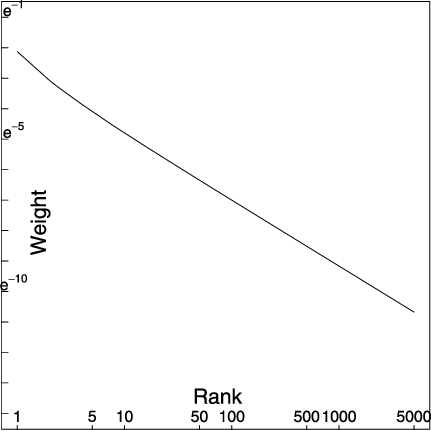}

\caption{Expected capital distribution curve for the pure hybrid model
in Example \protect\ref{ex: ExpCDC2}.} \label{fig: ExpCDC2}
\end{figure}
complement. The resulting expected capital distribution curve is
\textit{convex}; it is depicted in Figure \ref{fig: ExpCDC2}.
\end{example}
\begin{example}\label{ex: ExpCDC1}
Let us consider now a variant of this pure hybrid model, with a
variance structure that is observed in practice. The parameters are the
same as in Example \ref{ex: ExpCDC2}, except for the different choices
of the parameter $ c_{\ast} $ and for the rank-based variances $
\sigma_{k}^{2} := 0.075 + 6 k \times10^{-5} $ which are obtained
from the smoothed annualized values for 1990--1999 data as in Section
5.4, page 109 of \cite{F02} (see page 2319 of \cite{BFK05}). The
criterion from Corollary \ref{cor: CDC} cannot apply directly to this
case because the inequalities (\ref{eq: criterion CDC}) do not hold for
all $ {\mathbf p} \in\Sigma_{n} $. The expected capital distribution
curves under these parameters with (i) $ c_{\ast} = 0.02 $, (ii)
$ c_{\ast} = 0.03 $, (iii)~$ c_{\ast} = 0.04 $ are shown in
Figure \ref{fig: ExpCDC1}. The curve (i) is convex from the top rank to
%
%f3 ###
%
\begin{figure}

\includegraphics{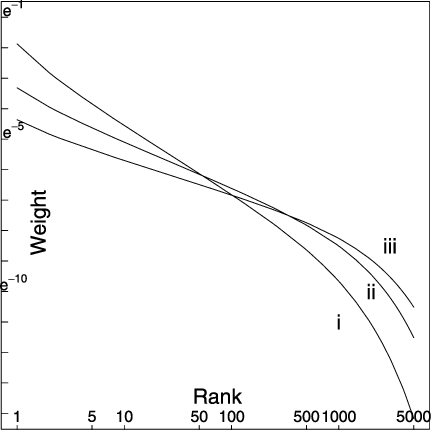}

\caption{Expected capital distribution curves for the hybrid model in
Example \protect\ref{ex: ExpCDC1}.}\vspace*{-3pt}
\label{fig: ExpCDC1}
\end{figure}
about the 25th rank, then turns concave until the lowest rank. The
other curves (ii) and (iii) behave similarly.
\end{example}
\begin{example} \label{ex: ExpCDC3}
Adopting the same parameter specifications in Example \ref{ex: ExpCDC1}(i)
$ c_{\ast} = 0.02 $, except the rank-based drift, that is, (iv)
the upwind first ranked stock $ g_{1} = -0.016 $, $ g_{k} = 0
, 2\le k \le n-1 $, $ g_{n} = (0.02)(2n-1)+0.016 $ and (v) the
windward top 50 stocks $ g_{1} = g_{2} = \cdots= g_{50} = -0.016
$, $ g_{k} = 0 , 51 \le k \le n-1 $, $ g_{n} = (0.02)(2n-1) +
0.8 $, we obtain concave curves as in Figure \ref{fig: ExpCDC3}. The
%
%f4 ###
%
\begin{figure}[b]

\includegraphics{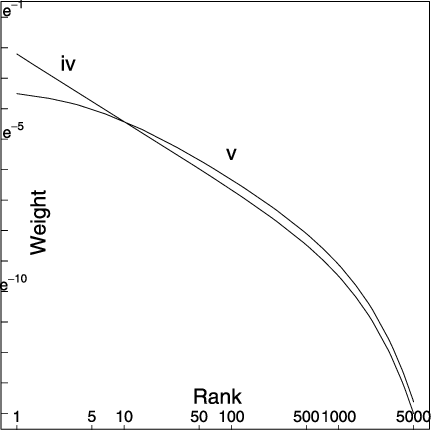}

\caption{Expected capital distribution curves for the hybrid model in
Example \protect\ref{ex: ExpCDC3}.}\vspace*{-3pt}
\label{fig: ExpCDC3}
\end{figure}
observed average curve and the estimated curve of the first-order Atlas
model for 1990--1999 (Figure 3 of \cite{BFK05}, page 2320) are concave.
The statistical inference for the capital distribution curves is an
interesting problem that we do not discuss here.
\end{example}

%%%%%%%%%%%%%%%%%%%%%%%%%%%%%%%%%%%%%%%%%%
%s6 ###
\section{Portfolio analysis} \label{sec: Portfolio Analysis}
%%%%%%%%%%%%%%%%%%%%%%%%%%%%

Let us consider investing in the market of (\ref{eq: model}) according
to a portfolio rule $ \Pi(\cdot)= (\Pi_{1}(\cdot), \ldots, \Pi
_{n}(\cdot))^{\prime} $. This is an $ \{\F_{t}\}$-adapted, locally
square-integrable process with
$ \sum_{i=1}^{n} \Pi_{i}(\cdot) = 1 $. Each\vspace*{1pt} $ \Pi_{i}( t) $
represents the proportion of the portfolio's wealth $ V^{\Pi}( t) $
invested in stock $i $ at time $t$, so
%
%e6.1 ###
%
\begin{equation} \label{eq: pflo rule}
\frac{d V^{\Pi}(t)}{V^{\Pi}(t)} = \sum_{i=1}^{n} \Pi_{i}(t) \cdot
\frac{d X_{i}(t)}{X_{i}(t)} ,\qquad V^{\Pi}(0) = w > 0 .
\end{equation}
For example, we may choose for every $ t \in[0,\infty) $ the vector
of market weights $ \mu_{i}( t) $, $ i = 1, \ldots, n $, as in
(\ref{RMW}). We shall call the resulting $ \Pi(\cdot) \equiv\mu
(\cdot) $ the \textit{market portfolio} and note $ V^{\mu}(\cdot) = w
X(\cdot)/X(0) $, thus from Proposition \ref{prop: long gr}: $ \lim
_{T\to\infty} (1 / T) \log V^{\mu}(T) \equiv\gamma$,
a.s.

For a constant-proportion portfolio $ \Pi(\cdot) \equiv\pi\in\Gamma
^{n}:= \{ (\pi_{1}, \ldots, \pi_{n})^{\prime} \in\R^{n} \vert\break
\sum_{i=1}^{n} \pi_{i}= 1 \} $ (which of course the market
portfolio is not), the solution of (\ref{eq: pflo rule}) is given by
%
%e6.2 ###
%
\begin{equation}
d \log V^{\pi}(t) = \gamma_{\pi}^{\ast}(t) \,d t + \sum_{i=1}^{n}
\pi_{i} d \log X_{i}(t) ,\qquad 0 \le t < \infty.
\end{equation}
Here we shall denote by $ ( a_{ij}(t))_{1\le i,j \le n} =
S(Y(t)) S(Y(t))^{\prime} $ the covariance process from (\ref{eq:
model matrix}), and introduce
%
%e6.3 ###
%
\begin{equation} \label{eq: excess grth}
\gamma_{\pi}^{\ast} (t):= \frac{1}{2} \Biggl( \sum_{i=1}^{n} \pi_{i}
a_{ii}(t) - \sum_{i,j=1}^{n} \pi_{i} a_{ij}(t) \pi_{j} \Biggr) ,\qquad
0 \le t < \infty,
\end{equation}
the \textit{excess growth rate} of the constant-proportion $ \Pi(\cdot)
\equiv\pi\in\Gamma^n $. Thus, for a constant-proportion portfolio we
can write the solution of (\ref{eq: pflo rule}), namely
%
%e6.4 ###
%
\begin{equation} \label{eq: const pflo}
V^{\pi}(t) = w\cdot\exp\Biggl[ \sum_{i=1}^{n} \pi_{i} \biggl\{ \frac
{A_{ii}(t)}{2} + \log\biggl(\frac{X_{i}(t)}{X_{i}(0)} \biggr) \biggr\} - \frac
{1}{2} \sum_{i,j=1}^{n} \pi_{i} A_{ij} (t) \pi_{j}
\Biggr]\hspace*{-28pt}
\end{equation}
as in (2.4) of \cite{J92}, where $ A_{ij} (\cdot) = \int^{\cdot}_{0}
a_{ij}(t) \,d t $ ; we set $ A(\cdot) := (A_{ij}(\cdot))_{1\le i,
j\le n} $.

%%%%%%%%%%%%%%
%s6.1 ###
\subsection{Target portfolio}
%%%%%%%%%%%%%%

Let us assume that, for every $ (t, \omega) \in[0, \infty) \times
\Omega$, there exists a vector $ \Pi^{\ast}(t, \omega) := (\Pi
_{1}^{\ast}( t, \omega), \ldots, \Pi_{n}^{\ast}(t, \omega))^{\prime}
\in\Gamma^n $ that attains the maximum of the wealth $ V^{\pi}( t,
\omega) $ over vectors $ \pi\in\Gamma^{n} $; and that the
resulting process $ \Pi^{\ast}(\cdot) $ defines a portfolio. Along
with Cover \cite{C91} and Jamshidian \cite{J92}, we shall call this $
\Pi^{\ast}(\cdot) $ a \textit{Target Portfolio}, and
%
%e6.5 ###
%
\begin{equation}
\label{TARP}
V_{\ast} (t) := \max_{\pi\in\Gamma^{n}} V^{\pi}(t) ,\qquad 0
\le t <\infty,
\end{equation}
the \textit{Target Performance} for the model. [The quantity of (\ref
{TARP}) is not necessarily equal to the performance $ V^{\Pi^*} (\cdot
) $ of the portfolio $ \Pi^{\ast}$.]

The Target Performance $ V_{\ast} (\cdot) $ exceeds the performance
of the leading stock, of the value-line index (the geometric mean), and
of any arithmetic average (such as the Dow Jones Industrial Average):
to wit, taking $ X_{1}(0) = \cdots= X_{n}(0) = 1 $, we have for
every vector $ (\alpha_{1}, \ldots, \alpha_{n})^{\prime} \in\Gamma
_{+}^{n} := \{ (\pi_{1} , \ldots, \pi_{n})^{\prime} \in\Gamma
^{n}\vert\pi_{i} \ge0 , i = 1, \ldots, n\} $ the almost sure
comparisons
%
%e6.6 ###
%
\begin{equation}
V_{\ast}(\cdot) \ge\max\Biggl[ \max_{1\le i \le n} X_{i}(\cdot)
, \Biggl(\prod_{j=1}^{n} X_{j}(\cdot)\Biggr)^{1 / n} ,
\sum_{j=1}^{n} \alpha_{j} X_{j}(\cdot) \Biggr] .
\end{equation}

Under the assumptions of Theorem \ref{prop: 1}, the limits $ \theta
_{{\mathbf p}} $ of the average occupation times in (\ref{eq: theta p
def}) exist almost surely, and so do the limits of the average
covariance rate $
\mfa_{ij}^{\infty} := \lim_{T\to\infty} A_{ij} (T) / T $; therefore,
$ \mfa^{\infty} := ( \mfa^{\infty}_{ij} )_{1\le i, j \le n}
$ is
%
%e6.7 ###
%
\begin{eqnarray}
\label{eq: 6.6.a}
\mfa^{\infty} & = & \lim_{T\to\infty} \frac{1}{T} \int^{T}_{0} (
a_{ij}(t)) _{1\le i, j \le n} \,d t \nonumber\\[-8pt]\\[-8pt]
&=& \lim_{T\to\infty} \frac{1}{T} \int^{T}_{0} \sum_{{\mathbf p}\in
\Sigma
_{n}} {\mathbf1}_{\RR_{\mathbf p}}(Y(s)) \cdot\mfs_{\mathbf p} \mfs
_{\mathbf p}^{\prime} \,d t = \sum_{{\mathbf p} \in\Sigma_{n}} \theta
_{{\mathbf p}}
\mfs_{\mathbf p} \mfs_{\mathbf p}^{\prime} ,
\nonumber
\end{eqnarray}
with $ \mfs_{\mathbf p} $ defined in (\ref{eq: vol cond 1}). It
follows from (\ref{eq: const pflo}) and Proposition \ref{prop: long gr}
that the asymptotic
long-term-average growth rate of a constant-proportion portfolio $
\pi\in\Gamma^{n} $ is
%
%e6.8 ###
%
\begin{equation} \label{eq: 6.7}\quad
\lim_{T\to\infty} \frac{1}{T} \log V^{\pi}(T) = \gamma+ \frac
{1}{2}\Biggl( \sum_{i=1}^{n}\pi_{i} \mfa^{\infty}_{ii} - \sum_{i,j=1}^{n}
\pi_{i} \mfa_{ij}^{\infty} \pi_{j} \Biggr) =: \gamma+ \gamma_{\pi
}^{\infty} .
\end{equation}
Maximizing this expression over $ \pi\in\Gamma^{n} $ amounts to
maximizing, over constant-proportion portfolios,
the excess growth rate
\[
\gamma_{\pi}^{\infty} = \frac{1}{2}
\Biggl( \sum_{i=1}^{n}\pi_{i} \mfa^{\infty}_{ii} -
\sum_{i,j=1}^{n} \pi_{i} \mfa_{ij}^{\infty} \pi_{j} \Biggr)
\]
that corresponds to the asymptotic covariance structure.

We shall call \textit{Asymptotic Target Portfolio} a vector $ \bar\pi=
(\bar\pi_{1}, \ldots, \bar\pi_{n})^{\prime} \in\Gamma^{n} $ that
attains $ \max_{\pi\in\Gamma^{n}} \gamma_{\pi}^{\infty} $. We can
regard this portfolio as \textit{asymptotic growth-optimal} over all
constant-proportion portfolios, in the sense that $
\lim_{T \rightarrow\infty} ( 1 / T)\times\log( V^{ \pi} (T) /
V^{\bar{\pi}} (T) ) \le0 $ holds a.s. for every $ \pi\in\Gamma
^n $.
\begin{example}
When there is no covariance structure by name, that is, $ \rho_{i,
j}\equiv0 $ for every $ 1\le i, j \le n $, we have $
A_{ij}(\cdot) \equiv0 $ for $ i \neq j $ in accordance with
(\ref{eq: vol cond 1}), (\ref{eq: model matrix}). In this case, we
compute a target portfolio $ \Pi^{\ast}(\cdot) $ as
%
%e6.9 ###
%
\begin{eqnarray}\qquad
\Pi^{\ast}_{i}(t) &=& \Biggl( 2 A_{ii}(t) \sum_{j=1}^{n} \frac
{1}{A_{jj}(t)} \Biggr)^{-1} \Biggl[ 2 - n - 2 \sum_{j=1}^{n} \frac
{1}{A_{jj}(t)}\log\biggl( \frac{X_{j}(t)}{X_{j}(0)} \biggr) \Biggr]
\nonumber\\[-8pt]\\[-8pt]
&&{} + \frac{1}{2} + \frac{1}{A_{ii}(t)} \log\biggl( \frac
{X_{i}(t)}{X_{i}(0)} \biggr),\qquad
i = 1, \ldots, n ,\nonumber
\end{eqnarray}
and an asymptotic target portfolio by
%
%e6.10 ###
%
\begin{equation} \label{eq: asy tar pflo}
\bar\pi_{i} = \frac{1}{2} \Biggl[ 1 - \frac{n-2}{\mfa_{ii}^{\infty
}} \Biggl( \sum_{j=1}^{n} \frac{1}{\mfa_{jj}^{\infty}} \Biggr)^{-1}
\Biggr] = \lim_{t \rightarrow\infty} \Pi^{\ast}_{i}(t),\qquad i = 1,
\ldots, n \mbox{, a.s.}\hspace*{-28pt}
\end{equation}

This constant portfolio
$ \bar{\pi} $ has exactly the same long-term growth rate as the
target performance in (\ref{TARP}), in particular
%
%e6.11 ###
%
\begin{eqnarray} \label{eq: equiv target-asym}
\lim_{T \to\infty} \frac{1}{T} \log V_\ast(T) &=& \lim_{T\to\infty}\frac{1}{T}
\log V^{\bar{\pi}}(T)\nonumber
\\[-8pt]
\\[-8pt] &=& \gamma+ \sum_{i=1}^n
\frac{\mfa_{ii}^\infty}{2} \bar{\pi}_i (1 - \bar{\pi}_i)
\qquad\mbox{a.s.;}\hspace*{-30pt}
\nonumber
\end{eqnarray}
on the other hand, we see from (\ref{eq: 6.7}) that it outperforms the
overall market rather significantly over long time horizons, namely
%
%e6.12 ###
%
\begin{eqnarray}
\lim_{T\to\infty} \frac{1}{T} \log\biggl( \frac{V^{\bar\pi
}(T)}{V^{\mu}(T)} \biggr) &=& \frac{1}{2}\sum_{i=1}^{n} \bar\pi_{i} ( 1-
\bar\pi_{i}) \mfa_{ii}^{\infty} \nonumber\\
&=& \frac{1}{8} \Biggl[ \sum_{i=1}^{n} \mfa_{ii}^{\infty} - (n-2)^{2}
\Biggl( \sum_{j=1}^{n} \frac{1}{\mfa_{jj}^{\infty}} \Biggr)^{-1} \Biggr] \\
&\ge&\frac
{n-1}{2} \Biggl( \sum_{i=1}^{n} \frac{1}{\mfa_{ii}^{\infty} }
\Biggr)^{-1}\nonumber
\end{eqnarray}
a.s., from the arithmetic mean--harmonic mean inequality.
\end{example}

With Cover \cite{C91} and Jamshidian \cite{J92}, we shall say that
stock $ i $ is \textit{asymptotically active}, if for the expression
of (\ref{eq: asy tar pflo}) we have $ \bar\pi_{i} > 0 $; and that
the entire \textit{market is asymptotically active}, if all its stocks are
asymptotically active, that is, if $ \bar\pi\in\Gamma_{++}^{n} :=
\{ (\pi_{1} , \ldots, \pi_{n})^{\prime} \in\Gamma^{n} \vert\pi
_{i} > 0 , i = 1, \ldots, n \} $.
\begin{example} \label{ex: active}
A sufficient condition for asymptotic activity of the model with $ n
\ge3 $ under the condition of Theorem \ref{prop: linear growing
var}, is obtained from (\ref{eq: asy tar pflo}) as
%
%e6.13 ###
%
\begin{eqnarray} \label{eq: suf asym act 1}
\frac{1}{\mfa_{ii}^{\infty}} &<& \frac{1}{n-2} \Biggl( \sum_{\ell=1}^{n}
\frac{1}{\mfa_{\ell\ell}^{\infty}} \Biggr), \quad \mbox{or equivalently}
%
%e6.14 ###
%
\\
\label{eq: suf asym act 2}
\Biggl( \sum_{{\mathbf p} \in\Sigma_{n}} \sigma_{{\mathbf p}^{-1}(i)}^{2}
\prod
_{j=1}^{n-1} \lambda_{\mathbf{ p}, j}^{-1}\Biggr)^{-1} &<& \frac{1}{n-2}
\Biggl[ \sum_{\ell=1}^{n} \Biggl( \sum_{{\mathbf p} \in\Sigma_{n}} \sigma_{{\mathbf
p}^{-1}(\ell)} ^{2} \prod_{j=1}^{n-1} \lambda_{{\mathbf p}, j}^{-1}
\Biggr)^{-1} \Biggr]\hspace*{-28pt}
\end{eqnarray}
for every $ i=1,\ldots, n $, with $ \lambda_{\mathbf{ p}, j} $
defined in (\ref{eq: lambda}); recall (\ref{eq: 6.6.a}), (\ref{eq:
theta p}) and (\ref{eq: vol cond 1}). This is the case in the constant
variance model $ \sigma_{1}^{2} = \cdots= \sigma_{n}^{2} $. In
general, it seems that the drift and volatility coefficients have
nontrivial effects on the condition (\ref{eq: suf asym act 2}).
\end{example}

%s6.2 ###
\subsection{Universal portfolio}
The \textit{universal portfolio} of Cover \cite{C91} and Jamshidian
\cite
{J92} is defined as
\[
\widehat\Pi_{i}(t) := \frac{\int_{\Gamma_{+}^{n}} \pi_{i} V^{\pi
}(t) \,d \pi}{\int_{\Gamma_{+}^{n}} V^{\pi}(t) \,d \pi} ,\qquad
0 \le t < \infty, 1\le i \le n .
\]
It is constructed completely in terms of quantities, such as the $
V^{\pi}(\cdot) $ for constant-proportion portfolios $ \pi$, that
are observable: no model-specific knowledge is required for its
construction. As can be checked easily, the wealth process of this
portfolio is given by the ``performance-weighting''
\[
V^{\widehat\Pi}(t) = \frac{ \int_{\Gamma_{+}^{n}} V^{\pi} (t) \,d
\pi}{\int_{\Gamma_{+}^{n}} d \pi} ,\qquad 0 \le t < \infty
,
\]
yet another observable quantity. It follows from Theorem 2.4 of
Jamshidian \cite{J92} that the universal portfolio does not lag
significantly behind the target portfolio: its performance lag is only
polynomial in time under an asymptotically active model. To wit, there
exists then
a positive constant $ C $, such that
\[
\lim_{T\to\infty} \biggl( \frac{V^{\widehat\Pi}(T)}{V_{\ast}(T)} \cdot
T^{ (n-1) / 2} \biggr) = C
\]
holds almost surely, thus also
%
%e6.15 ###
%
\begin{equation}
\label{eq: 6.13.c}
\lim_{T\to\infty} \frac{1}{T} \log\biggl( \frac{ V^{\widehat\Pi
}(T)}{V^{\bar\pi}(T)}\biggr) = \lim_{T\to\infty} \frac{1}{T} \log\biggl(
\frac{ V^{\widehat\Pi}(T)}{V_{\ast} (T)} \biggr) = 0 .
\end{equation}

In the context of the hybrid model, under the assumptions of Theorem
\ref{prop: linear growing var} and of Example \ref{ex: active}, the
universal portfolio attains the long-term growth rate of the target
portfolio $ \Pi^{\ast} $ and of the asymptotic target portfolio $
\bar\pi$. These are precisely the characteristics that make the
universal portfolio interesting: it is constructed based entirely on
quantities which are completely observable, yet its long-term
performance matches that of $ V_{\ast}(\cdot) $ in (\ref{TARP}),
and thus exceeds the performance of any constant-proportion portfolio.

%%%%%%%%%%%%%%%%%%%
%s6.3 ###
\subsection{Growth-optimal portfolio}
%%%%%%%%%%%%%%%%%%%

We shall call \textit{growth-optimal} a portfolio $ \varpi(\cdot) $
that satisfies the inequality $
\lim_{T \rightarrow\infty} ( 1 / T) \log( V^{ \Pi} (T) /
V^{\varpi} (T) ) \le0 $ almost surely, for any portfolio $ \Pi
(\cdot) $.

In order to find such a growth-optimal portfolio under no-name based
correlation $ \rho_{i, j} \equiv0 $ for $ 1 \le i, j \le n $,
we need to maximize over $ \pi\in\Gamma^{n} $ the quantity
(growth rate)
%
%e6.16 ###
%
\begin{equation} \label{eq: 6.13.a}
\Gamma(t; \pi) := \sum_{i=1}^{n} \biggl( \widetilde\gamma_{i}(t) + \frac
{1}{2}a_{ii}(t) \biggr) \pi_{i} - \frac{1}{2} \sum_{i=1}^{n} a_{ii}(t) \pi
_{i}^{2} ,
\end{equation}
where $ \widetilde\gamma_{i}(t)= \sum_{{\mathbf p} \in\Sigma_{n}}
{\mathbf1}_{\RR_{{\mathbf p}}}(Y(t)) g_{{\mathbf p}^{-1}(i)} + \gamma
_{i} +
\gamma$ is the $i$th element of $ G(Y(t)) $ of (\ref{eq: model
matrix}) (cf. Problem 4.6, page 108 in Fernholz and Karatzas \cite
{FK09}). By the Lagrange multiplier method, we obtain a vector that
attains this maximum, as
%
%e6.17 ###
%
\begin{equation} \label{eq: 6.13.b}
\varpi_{i}(t) = \frac{1}{2} + \frac{\widetilde\gamma_{i}(t) + \overline
{\gamma} (t)}{a_{ii}(t)} ,\qquad i = 1, \ldots, n , 0 \le
t < \infty,
\end{equation}
where the constraint $ \sum_{i=1} ^{n} \varpi_{i}(t) = 1 $ is
enforced by the multiplier
\[
\overline{\gamma} (t) = \Biggl( \sum_{i=1}^{n} \frac{1}{a_{i i}(t)}
\Biggr)^{-1} \Biggl( 1 - \frac{n}{2} - \sum_{j=1}^{n} \frac{\widetilde
\gamma_{j}(t)}{a_{jj}(t)} \Biggr) .
\]
The growth rate $ \Gamma(t; \varpi) $ of this portfolio $ \varpi
(\cdot) $, in the notation of (\ref{eq: 6.13.a}), (\ref{eq: 6.13.b})
and using (\ref{eq: drift cond 1}), is
\[
\Gamma(t; \varpi) = \frac{n\gamma}{2} + \frac{1}{2} \sum_{i=1}^{n} \frac
{\widetilde\gamma_{i}^{2}(t)}{a_{ii}(t)} - \frac{ \overline{\gamma
}^{2}(t)} {2} \sum_{i=1}^{n} \frac{1}{a_{ii}(t)}+ \frac{1}{ 8 } \sum
_{i=1}^{n} a_{ii}(t) .
\]

$\bullet$
In order to make some comparisons, let us specialize to the
equal-variance case, that is, $ \sigma_{1}^{2} = \cdots= \sigma
_{n}^{2} = \sigma^{2} $ with no name-based correlations $ \rho_{i,
j} \equiv0 $; we obtain under these assumptions
the expression
%
%e6.18 ###
%
\begin{equation} \label{eq: 6.13.d}
\varpi_{i}(t) = \frac{1}{n} + \frac{1}{\sigma^{2}} \Biggl( \gamma_{i}
+ \sum_{k=1}^{n} g_{k} {\mathbf1}_{Q^{(i)}_{k}}(Y(t)) \Biggr),\qquad
i = 1, \ldots, n
\end{equation}
for the growth-optimal portfolio, and
%
%e6.19 ###
%
\begin{eqnarray} \label{eq: asy wealth 1}\qquad
\lim_{T\to\infty} \frac{1}{T} \log V^{\varpi}(T) &=& \lim_{T\to\infty}
\frac{1}{T} \int^{T}_{0} \Gamma(t; \varpi)\,d t \nonumber\\[-8pt]\\[-8pt]
&=&\gamma+ \frac{\sigma^{2}}{2} \biggl( 1 - \frac{1}{n} \biggr) + \frac
{1}{2\sigma^{2}}\Biggl( \sum_{k=1}^{n} g_{k}^{2} - \sum_{i=1}^{n} \gamma
_{i}^{2}\Biggr)
,\nonumber
\end{eqnarray}
and from (\ref{eq: occup time 1}), (\ref{eq: occup ident}) we obtain
\[
\lim_{T\to\infty} \frac{1}{T} \int^T_0 \varpi_i(t) \,d t =
\frac{1}{n} + \frac{1}{\sigma^2} \Biggl( \gamma_i +\sum_{k=1}^n g_k \theta
_{k,i} \Biggr) = \frac{1}{n} = \bar{\pi_i} ,\qquad
i = 1, \ldots, n ,
\]
almost surely. On the other hand, from (\ref{eq: 6.13.c}),
(\ref{eq: asy tar pflo}) and (\ref{eq: 6.7}) we see that the universal
portfolio $ \widehat{\Pi}(\cdot) $ and the asymptotic target
portfolio $ \overline{\pi}_{i} = \frac{1}{n} $, $ i = 1, \ldots,
n $, have the same long-term growth rate, namely
%
%e6.20 ###
%
\begin{equation} \label{eq: asy wealth 2}
\lim_{T\to\infty} \frac{1}{T} \log V^{\bar\pi} (T) = \lim_{T\to\infty
} \frac{1}{T} \log V^{\widehat\Pi}(T) = \gamma+ \frac{\sigma^{2}}{2}
\biggl( 1 - \frac{1}{n} \biggr) .
\end{equation}
Under the conditions of (\ref{eq: drift cond 1}) and (\ref{eq: drift
cond 2}), we can verify
%
%e6.21 ###
%
\begin{equation} \label{ineq: g 2 k}
\sum_{k=1}^{n} g_{k}^{2} > \sum_{i=1}^{n} \gamma_{i}^{2} .
\end{equation}
To show (\ref{ineq: g 2 k}), we may assume without loss of generality
$ \gamma_{1} \ge\cdots\ge\gamma_{n} $ and hence that there
exists $ (\delta_{1}, \ldots, \delta_{n-1})^{\prime} \in(\R
_{+})^{n-1} \setminus\{0\} $ such that $ g_{k} = - (\gamma_{k} +
\delta_{k}) $ for $ k = 1, \ldots, n-1 $, and $ g_{n} = -
\gamma_{n} + (\delta_{1} + \cdots+ \delta_{n-1}) $ for (\ref{eq:
drift cond 1}) and (\ref{eq: drift cond 2}). Then we obtain
\begin{eqnarray*}
\sum_{k=1}^{n} g_{k}^{2} & = & \sum_{i=1}^{n-1} (\gamma_{i} + \delta
_{i})^{2} + \bigl(-\gamma_{n} + (\delta_{1} + \cdots+ \delta_{n-1})
\bigr)^{2} \\
& = & \sum_{i=1}^{n} \gamma_{i}^{2} + \sum_{i=1}^{n-1}
\bigl( \delta_{i}^{2} + 2 \delta_{i} (\gamma_{i} - \gamma_{n} ) \bigr)
+ \Biggl( \sum_{i=1}^{n-1} \delta_{i} \Biggr)^2
> \sum_{i=1}^{n} \gamma_{i}^{2} .
\end{eqnarray*}

We observe from (\ref{eq: 6.13.d})--(\ref{ineq: g 2 k})
that the growth-optimal portfolio $ \varpi(\cdot)
$ dominates in the long run both the universal portfolio $ \widehat
\Pi(\cdot) $ and the asymptotic target portfolio $ \bar\pi$,
a.s. The advantage of the universal portfolio is that it can be
constructed with total oblivion as to what the actual values of the
parameters of the model might be; some of these may be quite hard to
estimate in practice. By contrast, constructing the growth-optimal
portfolio $ \varpi(\cdot) $ as in (\ref{eq: 6.13.d}) requires
knowledge of all the model parameters, and keeping track of the
positions of all stocks in all ranks at all times.

%%%%%%%%%%%%%%%%%%%%%%%%%%%%%%%%%%%%%%%%%%%%%%%%%%%%%%%%%%%%%%%%%%%%%%%%%%%%%%%%%%%%%%%%%
\begin{appendix}\label{sec: Appendix}
%s7 ###
\section*{Appendix}
%%%%%%%%%%%%%%%%%%%%%%%%%%%%%%%%%%%%%%%%%%
%s7.1 ###
\subsection{\texorpdfstring{Preparations for the proof of Lemma \protect\ref{lm: for
prop 2}}{Preparations for the proof of Lemma 1.}} \label{app: prf of lm for prop 2}
The stochastic exponential
\[
\zeta(t) = \exp\biggl[ - \int^{t}_{0} \langle\xi(u) , d
W(u)\rangle- \frac{1}{2} \int^{t}_{0} \lVert\xi(u) \rVert^{2} \,d
u \biggr],\qquad 0 \le t < \infty,
\]
is a continuous martingale, where $ \xi(t) := S^{-1}(Y(t)) G(Y(t))
$ for $ 0 \le t < \infty$ and $ \lVert x \rVert^{2} := \sum
_{j=1}^{n}x_{j}^{2} , x \in\R^{n} $, and $ \langle x, y
\rangle= \sum_{j=1}^{n} x_{j} y_{j} , x, y \in\R^{n} $. Recall
that $ S(\cdot) $, $ S^{-1}(\cdot) $ and $ G(\cdot) $ in
(\ref{eq: model matrix}) are bounded. By Girsanov's theorem
\[
\widetilde W(t):= W(t) + \int^{t}_{0} S^{-1} (Y(u)) G (Y(u)) \,d u
,\qquad 0 \le t < \infty,
\]
is an $n$-dimensional Brownian motion under the new probability measure
$ \QQ$, locally equivalent to $ \PP$, that satisfies
%
%e7.1 ###
%
\setcounter{equation}{0}
\begin{equation} \label{eq: def meas Q}
\QQ(C) = \E^{\PP}(\zeta(T) {\mathbf1}_{C}) ,\qquad C \in\F_{T} , 0
\le T < \infty.
\end{equation}
Thus, equation (\ref{eq: model matrix}) under $ \PP$ is reduced to
%
%e7.2 ###
%
\begin{equation} \label{eq: model chg meas}
d Y(t) = S(Y(t)) \,d \widetilde W(t) ,\qquad 0 \le t < T \mbox{,
under } \QQ.
\end{equation}

%s7.1.1 ###
\subsubsection{Local time of Bessel processes}

Let us denote the $ \delta$-dimensional Bessel process by $
{\mathfrak r}^{(\delta)}(\cdot) $ for $ \delta> 1 $
\[
\mathfrak r^{(\delta)}(t) = \mathfrak r^{(\delta)}(0) + \int^t_0
\frac{\delta- 1}{2{\mathfrak r^{(\delta)} (s)}} \,d s + \widetilde
{B}(t) ,\qquad
0 \le t < \infty,
\]
where $ \widetilde{B}(\cdot) $ is the standard Brownian motion.
Since it is a continuous semimartingale, there is a modification $
\Lambda_{{\mathfrak r}^{(\delta)}}(\cdot) $ of its local time
accumulated at the origin, defined by
\[
\Lambda_{\mathfrak r^{(\delta)}}(t) = \frac{1}{2} \biggl( \mathfrak
r^{(\delta)}(t) - \mathfrak r^{(\delta)}(0) - \int^t_0
\operatorname{sgn}\bigl(\mathfrak r^{(\delta)}(s)\bigr) \,d \mathfrak{r}^{(\delta)}(s)
\biggr),\qquad
0 \le t < \infty,
\]
where the function $\operatorname{sgn}$ is defined by $\operatorname{sgn}(x) =
1 $ if $ x > 0 $ and $\operatorname{sgn}(x) = -1 $ if $ x \le0 $.
When $ \delta\ge2 $, $ \mfr^{(\delta)}(\cdot) $ never hits
the origin, and its local time at the origin is identically equal to
zero. Thus let us consider the case $ 1 < \delta<2 $. By the
occupation times formula and the right continuity of the semimartingale
local time, we obtain
%
%e7.3 ###
%
\begin{equation} \label{eq: BES local time in zero}
\Lambda_{\mathfrak r^{(\delta)}}(t) = \lim_{\varepsilon\downarrow0}
\frac{1}{2\varepsilon} \int^t_0 {\mathbf1}_{\{0 \le\mathfrak{r}
^{(\delta)}(s) \le\varepsilon\}} \,d s \qquad\mbox{almost surely for }
0 \le t < \infty.\hspace*{-28pt}
\end{equation}
On the other hand, it can be shown from
Lemma 3.1 and equation (3f) of
Biane and Yor \cite{BY87}, and also form
pages 285--289 of Rogers and Williams \cite{RW00}
that there exists a finite limit
%
%e7.4 ###
%
\begin{equation} \label{eq: BESLoc 2}
\lim_{\varepsilon\downarrow0} \frac{1}{2\varepsilon^\delta} \int^t_0
{\mathbf1}_{\{0 \le\mathfrak{r}^{(\delta)}(s) \le\varepsilon\}}\,d s
\qquad\mbox{almost surely for } 0 \le t < \infty
\end{equation}
[see (\ref{eq: BESLoc 8}) below]. Combining this fact with (\ref{eq:
BES local time in zero}), there is no accumulation of local time at the
origin for the case $ 1 < \delta< 2 $. Therefore, we conclude that
the local time $ \Lambda_{\mfr^{(\delta)}}(\cdot) $of the $
\delta$-dimensional Bessel process $ \mfr^{(\delta)}(\cdot) $
accumulated at the origin is \textit{identically equal to zero},
%
%e7.5 ###
%
\begin{equation} \label{eq: no lt for BES(3)}
\Lambda_{\mathfrak r^{(\delta)}}(t) \equiv0,\qquad 0 \le t < \infty
,
\delta> 1 .
\end{equation}
\begin{pf*}{Proof of (\ref{eq: BESLoc 2}) (\textup{Abridged from \cite{BY87,RW00}})}
Given the $ \delta$-dimensional Bessel processes $ \mfr^{(\delta
)}(\cdot) $, there is a one-dimensional Bessel process
$ \mfr^{(1)}(\cdot) $ which starts at
$ \mfr^{(1)}(0) = (2-\delta)^{-(2-\delta)} (\mfr^{(\delta
)}(0))^{2-\delta} $ and
satisfies the following pathwise relation:
%
%e7.6 ###
%
\begin{eqnarray} \label{eq: BESLoc 3}
\mfr^{(\delta)}(t) &=& (2-\delta) \bigl(\mfr^{(1)}(A_{t})\bigr)^{{1}/({2-\delta
})} ,\qquad
A_{t} := \inf\{s \ge0\dvtx C_{s} \ge t\} , \nonumber\\[-8pt]\\[-8pt]
C_{t} :\!&=& \int^{t}_{0} \bigl(\mfr^{(1)}(s) \bigr)^{({2\delta-2})/({2-\delta})}
\,d s ,\qquad
0 \le t < \infty.\nonumber
\end{eqnarray}
(This time-change formula is obtained with the parameters $ \nu=
-1/2 $, $ q = 2-\delta$, $ p = \frac{2-\delta}{1-\delta} $,
$ -\frac{2}{p} = \frac{2\delta- 2}{2-\delta} > 0 $ in Proposition
XI.1.11 of \cite{RY99}, which is originally from Lemma 3.1 of \cite
{BY87}. The index $ \nu= \frac{1}{2} -1 $ corresponds to the
one-dimensional Bessel process and the index $ \nu q = \frac{\delta
}{2} - 1 $ corresponds to the $ \delta$-dimensional Bessel process.)
The stochastic clocks $ C_\cdot$ and $ A_{\cdot} $ in (\ref{eq:
BESLoc 3}) do not explode in a finite time because of the instantaneous
reflection of $ \mfr^{(1)}(\cdot) $. Substituting this relation, we
compute the occupation time
%
%e7.7 ###
%
\begin{eqnarray} \label{eq: BESLoc 4}
\int^{t}_{0} {\mathbf1}_{\{0 \le\mfr^{(\delta)}(s) \le\varepsilon\}} \,d
s &=& \int^{t}_{0}
{\mathbf1}_{\{0 \le(2-\delta)(\mfr^{(1)}(A_{s}))^{{1}/({2-\delta})}
\le
\varepsilon\}} \,d s\nonumber\\[-8pt]\\[-8pt]
&=& \int^{A_{t}}_{0}
{\mathbf1}_{\{0 \le(2-\delta)(\mfr^{(1)}(s))^{{1}/({2-\delta})} \le
\varepsilon\}} \,d C_{s}
.\nonumber
\end{eqnarray}
It follows from (\ref{eq: BESLoc 3}) that $ \frac{d C_{t}}{d t} =
(\mfr^{(1)}(t) )^{({2\delta-2})/({2-\delta})} $ and hence the
right-hand side of (\ref{eq: BESLoc 4}) becomes
\[
\int^{A_{t}}_{0}
{\mathbf1}_{\{0 \le(2-\delta)(\mfr^{(1)}(s))^{{1}/({2-\delta})} \le
\varepsilon\}} \cdot
\bigl(\mfr^{(1)}(s) \bigr)^{({2\delta-2})/({2-\delta})} \,d s,\qquad
0 \le t < \infty.
\]

By the occupation time formula for the one-dimensional Bessel process
$ \mfr^{(1)}(\cdot) $, this expression becomes
\[
2 \int
_{(0, \eta)}
y^{({2\delta-2})/({2 - \delta})}
\Lambda^{\mfr^{(1)}}_{A_{t}}(y) \,d y ,
\]
where $ \eta:= (\frac{\varepsilon}{2-\delta})^{2-\delta} $ and $
\Lambda^{\mfr^{(1)}}_t (y) $
is the local time
accumulated by $ \mfr^{(1)}(\cdot) $
at the level $ y \in[0 , \infty) $ over
the time interval $ [0, t] $.
Changing the variable from $ y $ to $ x = (2-\delta)y^{
{1}/({2-\delta})} $
with
$ d y = \frac{x^{1-\delta}}{(2-\delta)^{1-\delta}} \,d x $, we obtain
\begin{eqnarray*}
&&\int^{t}_{0} {\mathbf1}_{\{0 \le\mfr^{(\delta)}(s) \le\varepsilon\}}
\,d s \\
&&\qquad= 2 \int^{\infty}_{0} {\mathbf1}_{\{0 \le x \le\varepsilon\}} \cdot
\frac{x^{2\delta-2}}{(2-\delta)^{2\delta-2}} \cdot
\frac{x^{1-\delta}}{(2-\delta)^{1-\delta}} \cdot\Lambda^{\mfr^{(1)}}_{A_{t}}
\biggl(\frac{x^{2-\delta}}{(2-\delta)^{2-\delta}}
\biggr)\,
d x \\
&&\qquad= 2 \int^{\varepsilon}_{0}
\frac{x^{\delta-1}}{(2-\delta)^{\delta-1}} \cdot\Lambda^{\mfr^{(1)}}_{A_{t}}
\biggl(\frac{x^{2-\delta}}{(2-\delta)^{2-\delta}}
\biggr)\,
d x, \qquad 0 \le t < \infty.
\end{eqnarray*}

Now by $ A_t < \infty$, $ 0 \le t <\infty$, and
by the right continuity of $ y \mapsto
\Lambda^{\mfr^{(1)}}_\cdot(y) $, we obtain
%
%e7.8 ###
%
\begin{equation} \label{eq: BESLoc 8}
\Lambda^{\mfr^{(1)}}_{A_{t}}(0) = \lim_{\varepsilon\downarrow0}
\frac{\delta(2-\delta)^{\delta-1}}{2 \varepsilon^{\delta}}
\int^{t}_{0} {\mathbf1}_{\{0 \le\mfr^{(\delta)}(s) \le\varepsilon\}} \,d
s <\infty,\qquad 0 \le t < \infty.\hspace*{-28pt}
\end{equation}
Therefore, we conclude that (\ref{eq: BESLoc 2}) holds for $ 1<
\delta< 2 $.
\end{pf*}

%s7.1.2 ###
\subsubsection{Comparisons with Bessel processes}
Now let us fix integers $ 1 \le i < j < k \le n $. Under $ \QQ
$ in (\ref{eq: def meas Q}) we shall compare the rank gap process
\[
\eta(t):= \max_{\ell=i, j, k} Y_{\ell}(t) - \min_{m=i,j,k} Y_{m}(t)
\]
with a Bessel process of dimension $ \delta> 1 $, using Lemmata
\ref{lm: comparison 2} and \ref{lm: comparison 1} below.

We introduce the function $ g(y) := [ (y_{i} - y_{j})^{2} + (y_{j} -
y_{k})^{2} + (y_{k}-y_{i})^{2}]^{1/2} $ for $ y \in\R^{n} $ and
note the comparison $ \sqrt{3 } \eta(\cdot) \ge g ( Y(\cdot)) $.
An application of It\^o's rule to $ g(Y(\cdot)) $ yields the
semimartingale decomposition
%
%e7.9 ###
%
\begin{equation}
d g(Y(t)) = h(Y(t)) \,d t + d \Theta(t) ,\qquad
0 \le t < \infty,
\end{equation}
where we introduce the $ (n\times3) $ matrix $ D_{ijk} := (
d_{i} , d_{j} , d_{k} ) $ with $ (n\times1) $ vectors $
d_{i} := \mfe_{i} - \mfe_{j} $, $ d_{j} := \mfe_{j} - \mfe_{k} $,
$ d_{k} := \mfe_{k} - \mfe_{i} $, we denote by $ \mfe_{i} $,
$ i = 1, \ldots, n $, the $i$th unit vector in $ \mathbb{R}^n $, and
%
%e7.10 ###
%
\begin{eqnarray} \label{eq: def sem decomp}\qquad
h(y) :\!&=& \frac{(R(y) - 1) Q(y)}{2 g(y)} ,\qquad R(y) := \frac
{\operatorname{Tr}(D_{ijk}^{\prime} S(y)S^{\prime}(y) D_{ijk})}{Q(y)} ,
\nonumber\\
Q(y) :\!&=& \frac{y^{\prime} D_{ijk} D_{ijk}^{\prime} S(y)S(y)^{\prime}
D_{ijk} D_{ijk}^{\prime} y}{y^{\prime} D_{ijk} D_{ijk}^{\prime} y}
,\qquad
y \in\R^{n} \setminus\mathcal Z , \nonumber\\
\mathcal Z :\!&=& \{ y \in\R^{n} \vert g(y) = ( y^{\prime}
D_{ijk} D_{ijk}^{\prime} y ) = 0 \} , \\
\Theta(t) :\!&=& \int^{t}_{0} \biggl( \sum_{\ell=i, j, k} \frac{S^{\prime
}(y) d_{\ell} d_{\ell}^{\prime} y}{g(y)} \bigg\vert_{y = Y(s)} \biggr)
\,d \widetilde W(s) , \nonumber\\
\langle\Theta\rangle(t) &=& \int^{t}_{0} Q(Y(s)) \,d s,\qquad
0\le t < \infty
.\nonumber
\end{eqnarray}
Here note that under the assumption on (\ref{eq: vol cond 1}), and
because $ 3 D_{ijk} D_{ijk}^\prime= D_{ijk}\times D_{ijk}^\prime D_{ijk}
D_{ijk}^\prime$, we have
%
%e7.11 ###
%
\begin{equation} \label{ineq: Q}\qquad
Q(\cdot) = \frac{3 y^{\prime} D_{ijk} D_{ijk}^{\prime} S(\cdot
)S(\cdot)^{\prime} D_{ijk} D_{ijk}^{\prime} y}{y^{\prime} D_{ijk}
D_{ijk}^{\prime}D_{ijk} D_{ijk}^{\prime} y} \ge3\min_{{\mathbf p} \in
\Sigma_n} \min_{\ell=1,\ldots, n} {\tilde\lambda}_{\ell, {\mathbf p}} >0
\end{equation}
in $ \R^n \setminus\mathcal Z $, where $ {\tilde\lambda}_{\ell,
{\mathbf p}} , \ell= 1, \ldots, n $, are the eigenvalues of the
positive-definite matrices $ \mfs_{\mathbf p} \mfs_{\mathbf p}^\prime$ for
$ {\mathbf p} \in\Sigma_n $, and so $ \langle\Theta\rangle(\cdot)
$ is strictly increasing when $ Y(\cdot) \in\R^n \setminus
\mathcal Z $.
Now define the stopping time $ \tau_{u} := \inf\{ t \ge0 \vert
\langle\Theta\rangle(t)\ge u \} $, and note
\[
\mfG(u) := g(Y(\tau_{u})) = g(Y(0)) + \int^{\tau_{u}}_{0} h(Y(t)) \,d
t + \widetilde B(u) ,\qquad 0 \le u < \infty,
\]
where $ \widetilde B(u) := \Theta(\tau_{u}) $, $ 0 \le u < \infty
$, is a standard Brownian motion, by the Dambis--Dubins--Schwarz
theorem of time-change for martingales.
Note that $ 1 / [Q(Y(\tau_u))] = d \tau_u / d u $,
when $ Y (\tau_u) \in\R^n \setminus\mathcal Z $.
Thus, with $ \mfd(u):= R(Y(\tau_{u})) $, we can write
\[
d \mfG(u) = \frac{\mfd(u) - 1}{2 \mfG(u)} \,d u + d \widetilde
B(u),\qquad 0 \le u < \infty, \mfG(0) = g(Y(0)) .
\]

The dynamics of the process $ \mfG(\cdot) $ are comparable to those
of a Bessel process $ \mfr^{(\delta)}(\cdot) $ with dimension $
\delta$, generated by the same $ \widetilde B(\cdot) $ and started
at the same initial point $ g(Y(0)) $. Since $ S(\cdot)S(\cdot
)^{\prime} $ is positive definite under (\ref{eq: vol cond 1}) and
$ \mbox{rank }(D_{ijk}) = 2 $, the $ (3\times3) $ matrix $
D_{ijk}^{\prime} S(\cdot) S(\cdot)^{\prime}D_{ijk} $ is nonnegative
definite and the number of its nonzero eigenvalues is equal to $
\mbox{ rank}(D_{ijk}^{\prime}S(\cdot)S(\cdot)^{\prime} D_{ijk}) = 2
$. Let us denote by $ \bar\lambda_{\ell, {\mathbf p}} $, $ \ell= 1,
2, 3 $, the eigenvalues of
$ D_{ijk}^{\prime} \mfs_{{\mathbf p}}\mfs_{{\mathbf p}^{\prime}}
D_{ijk} $
for $ {\mathbf p} \in\Sigma_n $.
Then for $ R(\cdot) $ in (\ref{eq: def sem decomp}) we obtain
%
%e7.12 ###
%
\begin{equation}
\label{eq: def delta 0}
R(\cdot) \ge\delta_{0} := \min_{{\mathbf p} \in\Sigma_{n}} \biggl( \frac
{\sum_{\ell=1}^{3} \bar\lambda_{\ell, {\mathbf p}} }{\max_{1\le\ell\le
3} \bar\lambda_{\ell, {\mathbf p}}} \biggr) > 1 \qquad\mbox{in } \R
^{n} \setminus\mathcal Z ,
\end{equation}
and so $ \mfd(\cdot) \ge\delta_{0} > 1 $ when $ Y(\tau_{\cdot})
\in\R^{n} \setminus\mathcal Z $. By a comparison argument similar
to that in the proof of Lemma 2.1 of \cite{IK09}, we may show that $
\mfG(t) \ge\mfr^{(\delta_{0})}(t) $ for $ 0 \le t < \infty$
a.s. Since $ \sqrt{3}\eta(t) \ge g (Y(t)) =\mfG(\langle\Theta
\rangle(t)) $ implies $ \sqrt{3} \eta(t) \ge\mfr^{(\delta_{0})}
(\langle\Theta\rangle(t))
$ for $ 0 \le t < \infty$, a.s., we obtain the following result.
\begin{lm} \label{lm: comparison 2} For the process $ Y(\cdot) $ of
(\ref{eq: model chg meas}) with (\ref{eq: vol cond 1}), the
multiple $ \sqrt{3} \eta(\cdot) $ of the rank-gap process
dominates, a.s. under $ \QQ$, a time-changed Bessel process $
\tilde\mfr(\cdot) := \mfr^{(\delta_0)}(\langle\Theta\rangle(\cdot
)) $ with dimension $ \delta_{0} $ as in (\ref{eq: def delta 0})
\[
\QQ\bigl( \sqrt{3} \eta(t) \ge\tilde\mfr(t) , 0 \le t <
\infty\bigr) =1 .
\]
\end{lm}
\begin{lm} \label{lm: comparison 1} Under $ \QQ$, the rank-gap
process $ \eta(\cdot) $
satisfies $ \langle\eta\rangle(t) \le c_{1} t $, $ 0 \le t
< \infty$ a.s. for some constant $ c_{1} > 0 $ and the local
time $ \Lambda_{\eta}(\cdot) $ of $ \eta(\cdot) $ at the origin
is identically equal to zero, that is, $ \Lambda_\eta(\cdot) \equiv
0 $, a.s.
\end{lm}
\begin{pf}
In fact, since the diffusion
coefficient matrix $ S(\cdot) $ of $ Y(\cdot) $ in (\ref{eq:
model chg meas}) is bounded and positive definite under (\ref{eq: vol
cond 1}), there exists a constant $ c_{1} $ such that $ \langle
\eta\rangle(t) \le c_{1} t $ for $ 0 \le t < \infty$ a.s.
Moreover, from (\ref{ineq: Q}) and Lemma \ref{lm: comparison 2}, there
exists a constant $ c_{2} := \min_{\mathbf{p}\in\Sigma_n, \ell=1,
\ldots, n} \tilde\lambda_{\ell, {\mathbf p} } >0 $, such that $
\langle\Theta\rangle(t) \ge c_{2} t $ holds for $ 0 \le t
<\infty$ a.s. It follows from the representation of local times
(Theorem VI. 1.7 of \cite{RY99}) and (\ref{eq: no lt for BES(3)}) with
Lemma \ref{lm: comparison 2} that
%
%e7.13 ###
%
\begin{eqnarray} \label{eq: lt is zero eta}\qquad
\Lambda_{\eta}(t) &=& \lim_{\varepsilon\downarrow0}\frac
{1}{2\varepsilon} \int^{t}_{0} {\mathbf1}_{\{0 \le\eta(s) <
\varepsilon\}} \,d \langle\eta\rangle(s) \le\lim_{\varepsilon
\downarrow0} \frac{\sqrt{3}c_{1}}{2 \varepsilon} \int^{t}_{0}
{\mathbf1}_{\{ 0 \le\sqrt{3}\eta(s) < \varepsilon\}} \,d s
\nonumber\\
&\le&\lim_{\varepsilon\downarrow0} \frac{\sqrt{3}c_{1}}{2
\varepsilon} \int^{t}_{0} {\mathbf1}_{\{0 \le\tilde\mfr(s) < \varepsilon
\}} \,d s \le\lim_{\varepsilon\downarrow0} \frac{\sqrt{3}
c_{1}}{2 c_{2} \varepsilon} \int^{\langle\Theta\rangle(t)}_{0}
{\mathbf1}_{\{0 \le\mfr^{(\delta)} (u) < \varepsilon\}} \,d u
\\
&\le&\sqrt{3} c_{1} c_{2}^{-1} \Lambda_{\mfr^{(\delta)}}(\langle
\Theta\rangle(t)) \equiv0 ,\qquad 0 \le t < \infty.\nonumber
\end{eqnarray}
\upqed\end{pf}

%s7.2 ###
\subsection{\texorpdfstring{Proof of Lemma \protect\ref{lm: for prop 2}}{Proof of Lemma 1.}}
\label{app: 7.2}
Define an increasing family of events $ C_{T}:=\{\Lambda_{\eta}(t) >
0 \mbox{ for some } t \in[0, T] \} $, $ T \ge0 $. By Lemma
\ref{lm: comparison 1} we obtain $ \QQ(C_{\infty}) = 0 $ and $ 0
= \QQ(C_{\ell} ) = \PP(C_{\ell}) $ for $ \ell\ge1 $. Then $
\PP(\Lambda_{\eta}(t) > 0$ for some $t \ge0) = \PP(\bigcup
_{\ell=1}^{\infty} C_{\ell}) = \lim_{\ell=\infty} \PP(C_{\ell}) = 0
$. Thus the local time $ \Lambda_{\eta}(t) $ of the rank gap process
$ \eta(\cdot) $ for $ (Y_{i}(\cdot) , Y_{j}(\cdot) ,
Y_{k}(\cdot) ) $ is zero for $ 0 \le t < \infty$ a.s. under
$ \PP$.

Since the choice of $ i, j, k $ is arbitrary, there is no local
time generated by the rank gap process of any three coordinates. The
rank gap process of more than three coordinates [e.g., $\max_{\ell=h,
i,j,k} Y_{\ell}(\cdot) - \min_{m=h,i,j,k} Y_{m}(\cdot) $] dominates
that of any three sub-coordinates. Therefore, by a similar argument as
(\ref{eq: lt is zero eta}) and its consequence, any local time of rank
gap process of more than three coordinates is zero for $ 0 \le t <
\infty$ a.s. \textit{under} $ \PP$.

To establish (\ref{eq: rank proc 3}) from this and (\ref{eq: rank proc
2}), and thus complete the proof of Lemma \ref{lm: for prop 2},
consider any integers (ranks) $ 1 \le a \le\ell<m \le b \le n $
with $ b - a \ge2 $, and observe that we have almost surely
\begin{eqnarray*}
0 &\equiv& \Lambda^{a, b} (t) = \int_0^t {\mathbf1}_{ \{ Z_a (s) = Z_b (s)
\} } \,d \bigl( Z_a (s) - Z_b (s) \bigr)
\\
&=& \int_0^t {\mathbf1}_{ \{ Z_a (s) = Z_b (s) \} } \,d \bigl( Z_a (s) -
Z_\ell(s) \bigr) + \int_0^t {\mathbf1}_{ \{ Z_a (s) = Z_b (s) \} } \,d
\bigl( Z_\ell(s) - Z_m (s) \bigr)
\\
&&{} + \int_0^t {\mathbf1}_{ \{ Z_a (s) = Z_b (s) \} } \,d \bigl( Z_m (s) -
Z_b (s) \bigr)
\\
&=& \int_0^t {\mathbf1}_{ \{ Z_a (s) = Z_b (s) \} } \,d \bigl( \Lambda^{a,
\ell} (s) + \Lambda^{ \ell, m} (s) + \Lambda^{m, b} (s) \bigr) \\
&\ge& \int_0^t {\mathbf1}_{ \{ Z_a (s) = Z_b (s) \} } \,d \Lambda^{ \ell, m}
(s) \ge0 .
\end{eqnarray*}
The a.s. equality $ \int_0^t {\mathbf1}_{ \{ Z_a (s) = Z_b (s) \} } \,d
\Lambda^{ \ell, m} (s) = 0 $ follows readily from this, as does
\[
\int_0^t {\mathbf1}_{ \{ N_k (t) \ge3 \} } \Biggl( \sum_{\ell=k
+1}^{n} d \Lambda^{ k , \ell} (s) - \sum_{\ell= 1}^{k-1} d \Lambda^{
\ell, k } (s) \Biggr) = 0
\]
and thus (\ref{eq: rank proc 3}) as well.

%s7.3 ###
\subsection{\texorpdfstring{Proof of Lemma \protect\ref{lm: BAR char}}{Proof of Lemma 2.}}
\label{sec: prf BAR char}
For each $k = 1, \ldots, n-1 $ the local time $ \Lambda
^{k,k+1}(\cdot) $ is a continuous additive functional of $ (\Xi
(\cdot), \mfP_{\cdot}) $ with support in $ \mfF_{k} $, and the
expectation of $ \Lambda^{k,k+1}(t) $ with respect to the invariant
distribution $ \nu(\cdot, \cdot) $ is finite for $ t \ge0 $.

It follows from the theory of additive functionals \cite{AKR67} that
there is a finite measure $ \nu_{k}(\cdot, \cdot) $ on $ \mfF_{k}
\times\Sigma_{n} $ such that
%
%e7.14 ###
%
\begin{equation} \label{eq: add fun 1}\qquad
\frac{1}{T} \E_{\nu} \biggl[ \int^{T}_{0} g(\Xi(s), \mfP_{s}) \,d
\Lambda^{k,k+1}(s) \biggr] = \frac{1}{2} \int_{\mfF_{k} \times\Sigma
_{n}} g(z, {\mathbf p}) \,d \nu_{k}(z, {\mathbf p})
\end{equation}
for every bounded measurable function $ g \dvtx \mfF_{k} \times\Sigma
_{n} \mapsto\R$. Let us denote by $ \nu_{0k}(\cdot) = \nu
_{k}(\cdot, \Sigma_{n}) $ the marginal distribution on $ \mfF_{k}
$. The absolute continuity of $ \nu_{0k}(\cdot) $ with respect to
$ (n-1)$-dimensional Lebesgue measure is argued by localization and
the properties of Reflected Brownian motion as in Theorem 7.1, Lemmata
7.7 and 7.9 of \cite{HW87Stoch}.

Now, by an application of It\^o's rule, for $ f \in C^{2}_{b}((\R
_{+})^{n-1}) $ we obtain
\begin{eqnarray*}
f(\Xi(T)) &=& f(\Xi(0)) + \int^{T}_{0} \langle\nabla f(\Xi(s)), d
\zeta^{\mathrm{mart}}(s) \rangle\\
&&{}+ \sum_{k=1}^{n-1}
\int^{T}_{0}[ \mathcal D_{k} f ] (\Xi(s)) \,d
\Lambda^{k,k+1}(s) \\
&&{} + \int^{T}_{0} [ \mathcal A f ](\Xi(s),
\mfP_{s}) \,d s ,\qquad T \ge0 ,
\end{eqnarray*}
where $ \zeta^{\mathrm{mart}}(\cdot) $ is the martingale part of
$ \zeta(\cdot) $ and $ \mathcal D_{k} $ and $ \mathcal A $
are differential operators defined in (\ref{eq: Diff Oper 1}). Taking
expectations with respect to $ \PP$ and then integrating for the
initial values with respect to the stationary distribution $ \nu(\cdot
, \cdot) $ with Fubini's theorem and (\ref{eq: add fun 1}), we obtain
\[
0 = \frac{T}{2} \sum_{k=1}^{n-1}\int_{\mfF_{k}} [ \mathcal D_{k} f
] (z) \,d \nu_{0k}(z) + T \int_{(\R_{+})^{n-1} \times\Sigma
_{n}} [ \mathcal A f ] (z, {\mathbf p}) \,d \nu(z, {\mathbf p}) .
\]
Dividing by $ T > 0 $, we obtain the basic adjoint relationship
(\ref{eq: BAR}).

%%%%%%%%%%%%%%%%%%%%%%%%%%%%%%%%%%%%%%%%%%%%%%%%%%%%%%%%%%%%%%%%%%%%%%%%%%%%%%%%%%%%%%%%%
%s7.4 ###
\subsection{\texorpdfstring{A sanity check of Corollary \protect\ref{cor: aot 2}}{A sanity check of Corollary 4.}} \label
{sec: sanity check}
In this section we verify that the entities $ ( \theta_{k, i})_{1\le
i, k \le n} $ in (\ref{eq: theta p}) satisfy (\ref{eq: occup
ident}). Since $ \theta_{k, i} $ is homogeneous in the product $
\prod_{j=1}^{n-1} [-4 (\sigma_{j}^{2}+ \sigma_{j+1}^{2})^{-1}] $,
it suffices to show $ \sum_{k=1}^{n} \widetilde\theta_{k,i}
(g_{k} + \gamma_{i}) = 0 $ where we use the modifications $
\widetilde\theta_{k,i}:= \sum_{\{{\mathbf p}(k) = i\}} \widetilde\theta
_{\mathbf p} $,
\[
\widetilde\theta_{\mathbf p} := \Biggl( \sum_{{\mathbf q} \in\Sigma_{n}} \prod
_{j=1}^{n-1} \widetilde\lambda_{{\mathbf q} , j}^{-1} \Biggr)^{-1} \prod
_{j=1}^{n-1} \widetilde\lambda_{{\mathbf p},j}^{-1} ,\qquad
\widetilde\lambda_{{\mathbf p}, j} := \sum_{\ell=1}^{j} \bigl( g_{\ell}
+ \gamma_{{\mathbf p}(\ell)}\bigr)
\]
of $ (\theta_{k, i} , \theta_{\mathbf p}, \lambda_{{\mathbf p}, j})
$, $ 1\le i, j, k \le n $, $ {\mathbf p} \in\Sigma_{n} $, for
notational simplicity. Note that $ \widetilde\lambda_{{\mathbf p}, n} =
0 $ from (\ref{eq: drift cond 1}) for $ {\mathbf p} \in\Sigma_{n} $.

First, observe for $ \ell= 2, \ldots, n $ and $ i = 1, \ldots,
n $,
%
%e7.15 ###
%
\begin{equation} \label{eq: 2 of sanity check}\quad
\sum_{\{{\mathbf p} \dvtx \mathbf{ p}(\ell-1)= i\}} \widetilde\lambda
_{{\mathbf p}, \ell-1}\widetilde\theta_{\mathbf p} + \sum_{\{{\mathbf
p} \dvtx {\mathbf p}
(\ell)=i\}} (g_{\ell} + \gamma_{i}) \widetilde\theta_{\mathbf p} \\
= \sum_{\{{\mathbf p} \dvtx {\mathbf p} (\ell) = i\}} \widetilde\lambda
_{{\mathbf p}, \ell}\widetilde\theta_{\mathbf p} .
\end{equation}
In fact, for every $ i, \ell$ define another permutation $
\widetilde{\mathbf p} $ from a (fixed) permutation $ {\mathbf p} \in\{
{\mathbf q} \in\Sigma_{n} \dvtx {\mathbf q}(\ell-1) = i\} $ by
\[
\widetilde{\mathbf p}(k) := \widetilde{\mathbf p} (k ; {\mathbf p} ) =
\cases{
{\mathbf p} (k) , &\quad $k = 1, \ldots, \ell-2, \ell
+ 1, \ldots, n$, \cr
{\mathbf p} (\ell) , &\quad $k = \ell-1$, \cr
i , &\quad $k = \ell$,}
\]
which is\vspace*{1pt} obtained by exchanging $ (\ell-1) $st and $ \ell$th
elements of $ {\mathbf p} \in\{ {\mathbf q} \in\Sigma_{n} \dvtx {\mathbf q}
(\ell-1) = i \} $, and also\vspace*{1pt} define $ M := (\sum_{{\mathbf q} \in
\Sigma_{n}} \prod_{j=1}^{n-1}\widetilde\lambda_{{\mathbf q},j}^{-1} )^{-1} $
here. Then $ \widetilde\lambda_{{\mathbf p}, j} = \widetilde\lambda
_{\widetilde{\mathbf p}, j} $ for $ j \neq\ell-1 $ and hence the
left-hand side of (\ref{eq: 2 of sanity check}) is
\begin{eqnarray*}
&&\sum_{\{{\mathbf p}\dvtx {\mathbf p} (\ell-1) = i\}} \widetilde\lambda
_{{\mathbf p},\ell-1}\cdot M \prod_{j=1}^{n-1} \widetilde\lambda
_{{\mathbf p}, j}
^{-1} + \sum_{\{{\mathbf p} \dvtx {\mathbf p} (\ell) = i\}} (g_{\ell} +
\gamma
_{i}) M \prod_{j=1}^{n-1}\widetilde\lambda_{{\mathbf p}, j} ^{-1}\\
&&\qquad= \sum_{\{\widetilde{\mathbf p} \dvtx \widetilde{\mathbf p} (\ell) = i \}
} M
\prod_{j\neq\ell-1}^{n-1} \widetilde\lambda_{\widetilde{\mathbf p}, j} ^{-1}
+ \sum_{\{\widetilde{\mathbf p} \dvtx \widetilde{\mathbf p} (\ell) = i\}}
\bigl(g_{\ell} + \gamma_{\widetilde p(\ell)}\bigr) M \prod_{j=1}^{n-1}\widetilde
\lambda_{\widetilde{\mathbf p}, j} ^{-1}
\\
&&\qquad = \sum_{\{\widetilde{\mathbf p} \dvtx \widetilde{\mathbf p} (\ell) = i \}
} \bigl[
\widetilde\lambda_{\widetilde{\mathbf p}, \ell-1} + g_{\ell} + \gamma
_{\widetilde{\mathbf p}(\ell)}\bigr] \cdot M \prod_{j=1}^{n-1} \widetilde
\lambda_{\widetilde{\mathbf p}, j} ^{-1}
= \sum_{\{{\mathbf p} \dvtx {\mathbf p} (\ell) = i\}} \widetilde\lambda
_{{\mathbf p}, \ell}\widetilde\theta_{\mathbf p} ,
\end{eqnarray*}
which is the right-hand side of (\ref{eq: 2 of sanity check}). Now applying
(\ref{eq: 2 of sanity check}) for $ \ell= 2, \ldots, n $, we obtain
\begin{eqnarray*}
\sum_{k=1}^{n} (g_{k} + \gamma_{i}) \widetilde\theta_{k, i} &=& (g_{1}
+ \gamma_{i}) \widetilde\theta_{1, i} + (g_{2} + \gamma_{i})
\widetilde\theta_{2, i} + \sum_{k=3}^{n} (g_{k}+ \gamma_{i})
\widetilde\theta_{k, i} \\
&=& \sum_{\{ {\mathbf p} \dvtx {\mathbf p} (2) = i\}} \widetilde\lambda
_{{\mathbf p}, 2}
\widetilde\theta_{{\mathbf p}} + \sum_{k=3}^{n}(g_{k} + \gamma_{i})
\widetilde\theta_{k, i}\\
&=& \cdots= \sum_{\{ {\mathbf p} \dvtx {\mathbf p} (n) = i\}} \widetilde
\lambda
_{{\mathbf p},n} \widetilde\theta_{\mathbf p}
=0
\end{eqnarray*}
for $ i = 1, \ldots, n $, because $ \widetilde\lambda_{{\mathbf p},
n} = 0 $ for $ {\mathbf p} \in\Sigma_{n} $. Therefore, (\ref{eq:
occup ident}) is satisfied.
\end{appendix}

\section*{Acknowledgments}
We are thankful to Professors Toshio Yamada, Peter Bank, Constantinos
Kardaras, Erhan Bayraktar and also the participants for their helpful
comments and discussions at the 8th Ritsumeikan--Columbia--JAFEE
International Symposium on Stochastic Processes/Mathematical Finance,
in the seminars at Quantitative Products Laboratory in Berlin, at
Boston University, at Columbia University and at the University of
Michigan. We are also thankful to the referee for valuable suggestions,
and to the Associate Editor for a remark on the proof of
Lemma \ref{lm: for prop 2}.

% imsref loaded by lrinkeviciute, 2010-09-16 12:34:50
%

%
\printaddresses

\end{document}